\documentclass[a4paper,11pt]{amsart}
\usepackage{amsmath,amsthm,amssymb,amsfonts,enumerate,color,esint}
\usepackage[pdftex]{graphicx}
\usepackage{pgfplots}
    \pgfplotsset{compat=1.18} 
\usetikzlibrary{shapes.geometric}
\usetikzlibrary{hobby}
\usepackage{float}
\usepackage{subfig}

\usepackage{mathtools}

\usepackage[colorlinks=true, allcolors=blue]{hyperref}
\usepackage{amsrefs}

\newcommand{\N}{\mathbb{N}}

\newcommand{\R}{\mathbb{R}}
\newcommand{\eps}{\varepsilon}

\newcommand{\abs}[1]{\left|#1\right|}

\newcommand{\sgn}{\operatorname{sgn}}

\renewcommand{\div}{\operatorname{div}}
\newcommand{\supp}{\operatorname{supp}}

\newcommand{\dist}{\operatorname{dist}}

\newtheorem{thm}{Theorem}[section]
\newtheorem{prop}[thm]{Proposition}
\newtheorem{cor}[thm]{Corollary}
\newtheorem{lem}[thm]{Lemma}
\theoremstyle{definition}
\newtheorem{defn}[thm]{Definition}
\newtheorem{rem}[thm]{Remark}

\newtheorem{notation}[thm]{Notation}

\newtheorem{ex}[thm]{Example}

\numberwithin{equation}{section}

\allowdisplaybreaks

\author[M. G. Delgadino]{\href{https://zaytam.github.io/}{Mat\'ias G.~Delgadino}}

\address{Department of Mathematics\\
The University of Texas at Austin\\
2515 Speedway, Austin\\
TX 78712, United States of America}
\email{matias.delgadino@math.utexas.edu}

\author[M. Vaughan]{\href{https://maryvaughan.github.io/}{Mary Vaughan}}

\address{Department of Mathematics and Statistics\\
The University of Western Australia\\
35 Stirling HWY\\
Crawley WA 6009, Australia}
\email{mary.vaughan@uwa.edu.au}



\begin{document}

\title[Symmetrizations and uniqueness of solutions to nonlocal equations]{Continuous symmetrizations and uniqueness of solutions to nonlocal equations}

\begin{abstract}
We show that nonlocal seminorms are strictly decreasing under the continuous Steiner rearrangement. This implies that all solutions to nonlocal equations which arise as critical points of nonlocal energies are radially symmetric and decreasing. Moreover, we show uniqueness of solutions by exploiting the convexity of the energies under a tailored interpolation in the space of radially symmetric and decreasing functions. As an application, we consider the long time dynamics of a higher order nonlocal equation which models the growth of symmetric cracks in an elastic medium.
\end{abstract}

\maketitle

\section{Introduction}
In the recent paper \cite{CHVY}, Carrillo, Hittmeir, Volzone and Yao used continuous Steiner symmetrization to show that all critical points of
\begin{equation}\label{eq:energy1}
\mathcal{E}[\rho]=\underbrace{\frac{1}{p-1}\|\rho\|_{L^p(\R^n)}^p}_{\text{Local Repulsion}}+\qquad \underbrace{\frac{1}{2}\int_{\R^n}\int_{\R^n} \rho(x)\rho(y) W(x-y)\;dxdy}_{\text{Nonlocal Attraction}}  
\end{equation}
are radially symmetric and decreasing as long as $W:\R^n\to\R$ is isotropic and attractive, meaning $W(z)=w(|z|)$ with $w'>0$. Noticing that the nonlinear aggregation-diffusion equation
\begin{equation}\label{eq:aggregation}
    \partial_t\rho =\Delta \rho^p+\nabla\cdot(\rho\nabla W*\rho) \quad \mbox{in}~\R^n \times (0,T)
\end{equation}
is the gradient flow of $\mathcal{E}$ in \eqref{eq:energy1}, they were able to conclude that all steady states of \eqref{eq:aggregation} are radially symmetric and decreasing, see also \cites{CHMV,huang2022nonlinear} for the extension of this result to more singular potentials. Nonlocal attraction repulsion of interacting particle models have recently garnered a lot of attention in the mathematical community, noticing in particular the case of the Patlak-Keller-Segel model \cites{blanchet2006two,dolbeault2004optimal,yao2014asymptotic}. Under an appropriate scaling limit, these models converge towards the higher order degenerate Cahn-Hilliard equation  \cites{topaz2006nonlocal,delgadino2018convergence,carrillo2023competing,carrillo2023degenerate,elbar2023degenerate}, where the local repulsion potential is given by the Dirichlet energy, or the $H^1$ energy.

The main aim of this paper is to extend the methods in \cite{CHVY} to more singular nonlocal equations. More specifically, we consider the models that arise when the 
repulsion potential energy is given by a fractional Gagliardo semi-norm
\begin{equation}\label{eq:seminorm}
[f]_{W^{s,p}(\R^n)}^p :=  \int_{\R^n} \int_{\R^n} \frac{|f(x) - f(y)|^p}{|x-y|^{n+sp}} \, dx \, dy,
\end{equation}
with $s\in(0,1)$. In the case $p=2$, we denote $H^s:=W^{s,2}$. 

As an application, we study the long time behavior of the fractional thin film equation
\begin{equation}\label{eq:thinfilm}
\partial_t u - \div (u^m \nabla (-\Delta)^s u) = 0  \quad \mbox{in}~\R^n \times (0,T),
\end{equation}
where $m\in \R$ and $(-\Delta)^s$ denotes the fractional Laplacian of order $0 < 2s < 2$.
It was recently proved by Lisini that \eqref{eq:thinfilm} with $m=1$ is the 2-Wasserstein gradient flow of the square of the $H^s$ seminorm, up to multiplying by an explicit constant depending on dimension $n$ and $s \in (0,1)$, see \cite{Lisini}.
For $m\not=1$, interpreting \eqref{eq:thinfilm} is an open problem; we reference \cite{Dolbeault} for $m \in (0,1)$. 
This equation arises as a model for the propagation of symmetric hydraulic fractures in an elastic medium, see below for more details.

We finally bring attention to the fact that modelling attraction and repulsion isotropically does not necessarilly imply radial symmetry of steady state solutions. When the repulsion potential is nonlocal there are several examples of non-radial energy minimizers, see for instance \cite{kolokolnikov2011stability}.



\subsection{Symmetric decreasing rearrangements}

Symmetric rearrangements are invaluable tools in the study of symmetry of solutions to partial differential equations. Thanks to the famous inequalities of Riesz \cite{https://doi.org/10.1112/jlms/s1-5.3.162}, P\'olya--Szeg\"o \cite{polya1951isoperimetric}, Almgren--Lieb \cite{almgren1989symmetric}, see also \cites{LiebLoss, burchard2009short}, we know that the absolute minimizer of many physical energies needs to be radially symmetric and decreasing. Hence, it follows that ground state solutions associated to partial differential equations that arise as first variations of these energies need to be radially symmetric and decreasing. However, this does not imply directly the symmetry of non-minimizing critical points, if they exist.

Continuous symmetrizations provide a useful way to deal with critical points, see \cite{Kawohl}. The continuous Steiner symmetrization $f^\tau$, $0 \leq \tau \leq \infty$, is a continuous interpolation between the original function $f$ and its Steiner symmetrization, see Figure \ref{fig:symmetrization}. 
We write the precise definition with more details in Section \ref{sec:prelim}. In \cites{Brock95,Brock}, Brock used this method to show radial symmetry of any positive solution to the nonlinear $p$-Laplace equations. More recently, Carrillo, Hittmeir, Volzone, and Yao revisited this technique to show symmetry of steady states of isotropic aggregation equations \cite{CHVY}, see also Proposition \ref{prop:pastresult} below. We further mention that the continuous Steiner symmetrization is used in \cite{Topaloglu} to establish a discrete isoperimetric inequality in $\R^2$ for Riesz-type nonlocal energies. 

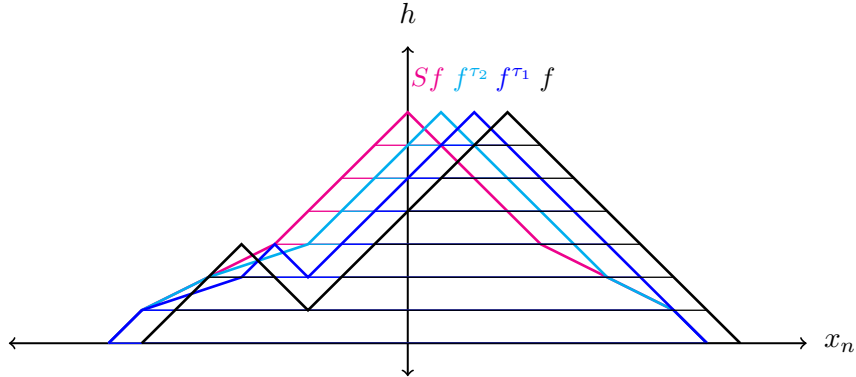
\begin{figure}[htbp]
 \begin{tikzpicture}[scale=1.75, use Hobby shortcut, closed=false]
 \tikzset{bullet/.style={diamond,fill,inner sep=1.25pt}}
\draw[line width=.75pt,<->] (0,0)--(0,2.5); 
	\draw (0,2.75) node {$h$};
\draw[line width=.75pt,<->] (-3,.25)--(3,.25); 
	\draw (3.25,.25) node {$x_n$};
\draw[line width=1pt, magenta] (-2.25,.25)--(-2,.5)--(-1.5,.75)--(-1,1)--(0,2)--(1,1)--(1.5,.75)--(2,.5)--(2.25,.25);
	\draw[magenta] (.15,2.25) node {\small $Sf$};
\draw[line width=.5pt, magenta] (-2.25,.25)--(2.25,.25); 
\draw[line width=.5pt,magenta] (-2,.5)--(2,.5);
\draw[line width=.5pt,magenta] (-1.5,.75)--(1.5,.75);
\draw[line width=.5pt,magenta] (-1,1)--(1,1);
\draw[line width=.5pt,magenta] (-.75, 1.25)--(.75,1.25);
\draw[line width=.5pt,magenta] (-.5, 1.5)--(.5,1.5);
\draw[line width=.5pt,magenta] (-.25, 1.75)--(.25,1.75);
\draw[line width=1pt, cyan] (-2.25,.25)--(-2,.5)--(-1.5,.75)--(-.75,1)--(.25,2)--(1.25,1)--(1.5,.75)--(2,.5)--(2.25,.25);
	\draw[cyan] (.48,2.25) node {\small $f^{\tau_2}$};
\draw[line width=.5pt, cyan] (-2.25,.25)--(2.25,.25); 
\draw[line width=.5pt,cyan] (-2,.5)--(2,.5);
\draw[line width=.5pt,cyan] (-1.5,.75)--(1.5,.75);
\draw[line width=.5pt,cyan] (-.75,1)--(1.25,1);
\draw[line width=.5pt,cyan] (-.5, 1.25)--(1,1.25);
\draw[line width=.5pt,cyan] (-.25, 1.5)--(.75,1.5);
\draw[line width=.5pt,cyan] (0, 1.75)--(.5,1.75);
\draw[line width=1pt, blue] (-2.25,.25)--(-2,.5)--(-1.25,.75)--(-1,1)--(-.75,.75)--(.5,2)--(1.75,.75)--(2,.5)--(2.25,.25);
	\draw[blue] (.8,2.25) node {\small $f^{\tau_1}$};
\draw[line width=.5pt, blue] (-2.25,.25)--(2.25,.25); 
\draw[line width=.5pt,blue] (-2,.5)--(2,.5);
\draw[line width=.5pt,blue] (-1.25,.75)--(-.75,.75);
\draw[line width=.5pt,blue] (-.75,.75)--(1.75,.75);
\draw[line width=.5pt,blue] (-.5,1)--(1.5,1);
\draw[line width=.5pt,blue] (-.25, 1.25)--(1.25,1.25);
\draw[line width=.5pt,blue] (0, 1.5)--(1,1.5);
\draw[line width=.5pt,blue] (.25, 1.75)--(.75,1.75);
\draw[line width=1pt] (-2,.25)--(-1.25,1)--(-.75,.5)-- (.75,2)--(2.5,.25);
	\draw (1.05,2.25) node {\small $f$};
\draw[line width=.5pt] (-2,.25)--(2.5,.25); 
\draw[line width=.5pt] (-1.75,.5)--(2.25,.5);
\draw[line width=.5pt] (-1.5,.75)--(-1,.75);
\draw[line width=.5pt] (-.5,.75)--(2,.75);
\draw[line width=.5pt] (-.25,1)--(1.75,1);
\draw[line width=.5pt] (0, 1.25)--(1.5,1.25);
\draw[line width=.5pt] (.25, 1.5)--(1.25,1.5);
\draw[line width=.5pt] (.5, 1.75)--(1,1.75);
\end{tikzpicture}
\caption{The continuous Steiner symmetrization $f^\tau$ for $0 < \tau_1 < \tau_2 < \infty$ as it interpolates between the function $f$ and its Steiner symmetrization $Sf$.}
\label{fig:symmetrization}
\end{figure}

Our first main result is that the Gagliardo seminorms are decreasing under continuous Steiner symmetrizations. 
It is well-known that the Gagliardo seminorms \eqref{eq:seminorm} are the natural energies associated to fractional $p$-Laplacians and thus are linked to free energies arising from fractional equations, such as \eqref{eq:thinfilm}.
They also arise in game theory \cite{Caffarelli}, anomalous diffusion \cite{KM}, minimal surfaces \cite{minimal}, to name only a few.
We reference the reader to \cite{Hitchhikers} and the references therein for more on fractional Sobolev spaces and fractional $p$-Laplacians.

\begin{thm}\label{thm:main2}
Let $0 < s < 1$ and $1 < p < \infty$.
For any positive $f\in L^1(\R^n)\cap C(\R^n)$ that is not radially decreasing about any center,  there are constants $\gamma = \gamma(n,s,p,f)>0$, $\tau_0 = \tau_0(f)>0$, and a hyperplane $H$ such that 
\begin{equation}\label{eq:main estimate}
 [f^{\tau}]_{W^{s,p}(\R^n)}^p \leq  [f]_{W^{s,p}(\R^n)}^p- \gamma \tau \quad \hbox{for all}~0 \leq \tau \leq \tau_0,
\end{equation}
where $f^{\tau}$ is the continuous Steiner symmetrization of $f$ about $H$.
\end{thm}

Our result extends to a more general class of kernels as highlighted in the next remark. 

\begin{rem}
As a direct consequence of the proof, Theorem \ref{thm:main2} holds for any fractional seminorm of the form
\[
\int_{\R^n} \int_{\R^n} \frac{|f(x) - f(y)|^p}{K(|x-y|)} \, dx \, dy
\]
where $K:[0,\infty) \to [0,\infty]$ is increasing. 
\end{rem}

To recover the corresponding local results as $s \to 1^-$ and $s\to 0^+$, one must normalize the energy by  multiplying \eqref{eq:main estimate} by $s(1-s)$, see \cites{BBM,MazyaShaposhnikova}. 
We will showcase in Remark \ref{rem:gamma-limit} that the constant $\gamma = \gamma(s)$ in Theorem \ref{thm:main2} remains strictly positive and bounded as $s\to 0^+,1^-$. 
Consequently, $s(1-s) \gamma \to 0$ as $s \to 0^+,1^-$. 
Regarding $s=0$, it is known that continuous Steiner symmetrizations preserve the $L^p$ norms, see \cite{CHVY}.
As for $s=1$, we have the following. 

\begin{cor}\label{cor:Lip}
Let $1 < p < \infty$. 
For any non-negative $f\in L^1(\R^n)\cap C(\R^n)$ that is not radially decreasing about any center,  there is a constant $\tau_0 = \tau_0(f)>0$ such that
\[
[f^{\tau}]_{W^{1,p}(\R^n)}^p \leq  [f]_{W^{1,p}(\R^n)}^p \quad \hbox{for all}~0 \leq \tau \leq \tau_0.
\]
Consequently,
\begin{equation}\label{eq:f-Lip}
 [ f^{\tau}]_{\text{Lip}(\R^n)} \leq   [f]_{\text{Lip}(\R^n)} \quad \hbox{for all}~0 \leq \tau \leq \tau_0.
\end{equation}
\end{cor}

The control on the Lipschitz norm in \eqref{eq:f-Lip} was first established by Brock for a different variant of continuous Steiner symmetrizations but can be adapted for our setting, see \cite{Brock95}*{Theorem 11}. 
We echo his observation in \cite{Brock95}*{Remark 8}
which states that Lipschitz continuity is the best regularity one can expect under continuous Steiner symmetrizations as kinks can form when symmetrizing a $C^1$ function that is not quasiconvex. The inequality in Corollary \ref{cor:Lip} is not strict as a simple counterexample can be constructed using that the norm is local, see Example \ref{ex:zero}. 


Since Steiner symmetrizations are rearrangements in $\R^n$ with respect to a single direction, the proof of Theorem \ref{thm:main2} relies on a corresponding one-dimensional result for characteristic functions (see Lemma \ref{lem:one-d}). 
In fact, the definition of continuous Steiner symmetrization of a set (and hence a function) is understood first in terms of open intervals, then finite unions of open intervals, and lastly infinite unions.  
Accordingly, we have found it insightful to make a special study of those functions whose level sets can be expressed as a finite union of open intervals, also known as \emph{good functions}. 
We establish an explicit version of Theorem \ref{thm:main2} for good functions. 
Here, we will explain the simplest setting and delay the detailed result until Section \ref{sec:good}. 

Let $p=2$ and consider a function with a simple geometry, that is, a positive function $f: \R \to \R$ whose level sets are each a single open interval. 
In particular, for each $h >0$, there are at most two solutions to $f(x) = h$ which we denote by $x_- = x_-(h)$ and $x_+ = x_+(h)$. 
Note that, if $f$ is radially decreasing, then each level set $(x_-,x_+)$ is centered at the origin. 
We use a continuous Steiner symmetrization with constant speed towards the origin:
\begin{equation}\label{eq:1D-defn}
x_{\pm}^\tau = x_{\pm} - \tau \sgn(x_+ + x_-) \quad \hbox{for all}~0 \leq \tau \leq \frac{|x_++x_-|}{2}. 
\end{equation}
The energy $[f^\tau]_{H^s(\R)}^2$ is a double integral involving $f^\tau$ over the spatial variables $x,y \in \R$. Formally, we can make a change of variables to write the energy instead as a double integral involving $x_{\pm}^\tau, y_{\pm}^\tau$ over the heights $h,u>0$ where $f^\tau(x_{\pm}) = h$ and $f^\tau(y_{\pm}) = u$ (see Lemma \ref{lem:energy for good}). 
With this,
we can write the derivative of the energy for $f^\tau$ in terms of the level sets of $f$ as
\begin{align*}
\frac{d[f^\tau]_{H^s(\R)}^2}{d \tau}\bigg|_{\tau=0}
	&=c_s
	 \int_{0}^\infty \int_{0}^\infty  (\sgn(x_++x_-)-\sgn(y_++y_-))\\ 
	 &\quad \bigg[\frac{\sgn(x_+-y_+)}{\abs{x_+-y_+}^{2s}}
	 -\frac{\sgn(x_+-y_-)}{\abs{x_+-y_-}^{2s}} 
	 -\frac{\sgn(x_--y_+)}{\abs{x_--y_+}^{2s}}
	 +\frac{\sgn(x_--y_-)}{\abs{x_--y_-}^{2s}}
	  \bigg] \, dh \, du.
\end{align*} 
In the integrand, we see the derivative in $\tau$ of  \eqref{eq:1D-defn} multiplied by an antiderivative of the kernel at the endpoints of the corresponding level sets.   
We will show in the proof of Proposition \ref{lem:deriv for good} that this product is negative when the level sets $(x_-,x_+)$ and $(y_-,y_+)$ are not centered. We should note that a different expression for the derivative can already be found in \cite{CHVY}*{Equation (2.23)} for more regular kernels. In Section~\ref{sec:good}, we also write an explicit formula for the derivative $\frac{d}{d\tau}|\nabla f^\tau|_{L^p}^p$, see Proposition~\ref{lem:good-local} and Corollary~\ref{cor:local deriv}.

\subsection{Uniqueness}

Uniqueness of critical points within the class of positive and fixed mass functions does not follow immediately from the fact that these are radially symmetric and decreasing. For the specific case of the nonlinear aggregation equation \eqref{eq:aggregation}, when $p\in(1,2)$ one can construct an ad hoc isotropic attractive interaction potential such that there are an infinite amount of steady state solutions see \cite{delgadino2022uniqueness}*{Theorem 1.2}. 
On the other hand, in the case $p\in[2,\infty)$, 
\cite{delgadino2022uniqueness} also shows uniqueness of critical points, by introducing a height function interpolation curve over radially symmetric profiles and showing that the associated energy \eqref{eq:energy1} is strictly convex under the interpolation.
See Section~\ref{sec:interp} for definitions and details.

We show that the square of the $H^s$ seminorms are strictly convex under the height function interpolation. 

\begin{thm}\label{thm:nonlocal-convexity}
Fix $0 < s < 1$. 
Let $f_0,f_1 \in C(\R^n)$ be two distinct, non-negative, symmetric decreasing functions with unit mass and let $\{f_t\}_{t\in[0,1]}$ be the height function interpolation between $f_0$ and $f_1$. 
Then 
\[
t \mapsto \|f_t\|^2_{H^s(\R^n)}   
\]
is strictly convex for all $0 < t < 1$. 
\end{thm}
The uniqueness of solutions to fractional Laplace equations is a deep and active area of research, see for instance \cites{frank2013uniqueness,frank2016uniqueness,CdMHMV,fracPM,fracPM2, Cabre-Sire, BSV,CaffarelliSilvestre}. Currently the methods to show uniqueness within the class of radially symmetric states are quite involved and at times only address the uniqueness of global minimizers and not of general critical points \cites{frank2013uniqueness,frank2016uniqueness}. The spirit of Theorem~\ref{thm:nonlocal-convexity} is to try to simplify the theory, when possible.

The uniqueness methods presented here do not cover the general Gagliardo seminorm $W^{s,p}$ for $p \not=2$. Still, we are able to show the convexity under the interpolant of $W^{1,p}$ seminorms, for $p\ge 2n/(n+1)$, see Proposition~\ref{prop:interW1p}. Moreover, we also cover the case of the potential energy when the potential is radial and increasing, which we use in the next section, see Proposition~\ref{prop:V}.

\subsection{Application to fractional thin film equations}

As an application of Theorem \ref{thm:main2} and Theorem \ref{thm:nonlocal-convexity}, we study the uniqueness of stationary solutions and long time asymptotic of fractional thin film equations given by \eqref{eq:thinfilm}. The fractional thin film equation with exponent $s=1/2$ and mobility $m=3$, 
was originally derived to model the growth of symmetric hydraulic fractures in an elastic material arising from the pressure of a viscous fluid pumped into the opening, see the original references \cites{geertsma1969rapid,zheltov1955hydraulic}.
A practical man-made application of this phenomenon is commonly known as fracking, which enhances oil or gas extraction from a well. In nature, this process occurs in volcanic dikes when magma causes fracture propagation through the earth’s crust, and also when water opens fractures in ice shelf.

Nonetheless, due to the nonlocal and higher order nature of this equation,  there is a striking lack of mathematical analysis regarding solutions to fractional thin film equations. Indeed, \eqref{eq:thinfilm} is an interpolation between the second order porous medium equation ($s=0$), see \cite{vazquez2007porous}, and the fourth order thin film equation ($s=1$), see \cites{bertozzi1994lubrication,otto1998lubrication}. We mention that the study of self-similar solutions was first started by Spence and Sharp \cite{spence1985self}, but rigorous existence of solutions was only recently shown by Imbert and Mellet in \cites{ImbertMellet11,ImbertMellet15}. 

Even more recently, Segatti and Vazquez \cite{SegattiVazquez}, studied the long time behavior of \eqref{eq:thinfilm} with linear mobility $m(u)=u$, by studying the Barenblatt re-scalings \cite{Barenblatt}. Namely, if $u$ is a solution to \eqref{eq:thinfilm}, we consider $v$ defined by the re-scaling
\[
    u(x,t) = \frac{1}{(1+t)^{\alpha}} v\bigg(\frac{x}{(1+t)^{\beta}},\log(1+t)\bigg) \quad \hbox{in}~\R^n \times (0,T)
\]
where $\alpha, \beta>0$ are given explicitly by
\[
\alpha = \frac{n}{n+2(1+s)}, \quad \beta = \frac{1}{n+2(1+s)}.
\]
The function $v = v(y,\tau)$ satisfies the re-scaled equation
\begin{equation}\label{rescaled}
    \partial_\tau v=\nabla \cdot\left(v \nabla_y\left((-\Delta)^sv+\beta \frac{|y|^2}{2}\right)\right)
\end{equation}
which contains an extra confining term. Under an extra qualitative assumption on the integrability of the gradient, Segatti and Vazquez showed \cite{SegattiVazquez}*{Theorem 5.9}  that $v$ converges as $\tau\to\infty$ to a solution of
\begin{equation}\label{eq:stationary}
\begin{cases}
(-\Delta)^s v = \sum_i \lambda_i \chi_{\mathcal{P}_i}(y) - \frac{\beta}{2} |y|^2 & \hbox{in}~\supp(v) \subset \R^n\\
v \geq 0 & \hbox{in}~\R^n
\end{cases}
\end{equation}
where $\mathcal{P}_i$ are the connected components of $\supp(v)$ and $\lambda_i$ are the corresponding Lagrange multipliers, which can change from one connected component to another.

As is the case for higher order equations, like for instance the classical thin film equation, uniqueness results that do not assume strict positivity of the functions are rare. We mention the work  Majdoub, Masmoudi and Tayachi \cite{majdoub2018uniqueness} as one of the few available examples on uniqueness of source solutions to the thin film equation. With respect to the problem at hand, Segatti and Vazquez showed the solution to \eqref{eq:stationary} is unique under the extra assumption that the solution has a single connected component. In this work, we instead use the re-arrangement techniques described above to show that the solution to \eqref{eq:stationary} is first radially symmetric and then unique, by using Theorem~\ref{thm:main2} and Theorem~\ref{thm:nonlocal-convexity} respectively.

\begin{thm}\label{thm:uniqueness}
Fix $0 < s < \frac{1}{2}$ and let $v\in C^{0,1}(\R^n)$ be a compactly supported solution to \eqref{eq:stationary}, then $v$ is radially decreasing. Moreover, up to scaling, it is uniquely given by
\begin{equation}\label{eq:explicitv}
v(x) = \frac{1}{\lambda^{s}\kappa} (1-\lambda|x|^2)_+^{1+s}    
\end{equation}
where $\lambda>0$ and $\kappa = 4^s \Gamma(s+2)\Gamma(s + n/2)/\Gamma(n/2)$.
\end{thm}
\begin{rem}
    The function $v$ in \eqref{eq:explicitv} belongs to $C^{1+s}$, hence the Lipschitz assumption in Theorem~\ref{thm:uniqueness} is natural, but it is not currently known.
\end{rem}

Following the recent work of Lisini \cite{Lisini} (see also the work of Otto \cite{otto1998lubrication}), 
\eqref{rescaled} is the 2-Wasserstein gradient flow of the following energy functional
\begin{equation}\label{eq:energyyyy}
\mathcal{E}(v) =
 c_{n,s}[v]_{H^s(\R^n)}^2
	 +  \frac{\beta}{2}\int_{\R^n}|y|^2 v(y) \, dy.
\end{equation}
where $c_{n,s} \simeq s(1-s)>0$ is a normalizing constant, depending only on $n \in \N$ and $s \in (0,1)$, such that 
$c_{n,s}[v]_{H^s}^2 =\langle (-\Delta)^s v, v \rangle_{L^2(\R^n)}$. 
Hence, any steady state is a critical point of \eqref{eq:energyyyy}. 
By Theorem \ref{thm:main2}, we can show that if $v$ is not radially decreasing, then \eqref{eq:energyyyy} is decreasing to first order under continuous Steiner symmetrization.  However, notice from Figure \ref{fig:symmetrization} that $f^\tau$ does not necessarily preserve the support of $f$, which we need for the proof of Theorem \ref{thm:uniqueness} since \eqref{eq:stationary} is only satisfied in $\supp(v)$. To address this issue, we slow down the speed of the level sets near the base of the solution $v$. 
In particular,
for $h>0$, if $x_{\pm} = x_{\pm}(h)$ are the boundary of the super-level set $\{f>h\}$, then 
we replace \eqref{eq:1D-defn} with
\[
x_{\pm}^\tau = x_{\pm} - \tau  \sgn(x_++x_-)\min\left\{1, \frac{h}{h_0}\right\},
\]
for some small, fixed $h_0>0$. 
Unlike \eqref{eq:1D-defn}, the perturbation only makes sense for super-level sets. 
The regularity assumption on $v$ assures that the perturbation is well-defined, as in general level sets can fall, see Figure~\ref{fig:truncated}. This idea first appeared in the work of Carrillo, Hittmeier, Volzone, and Yao \cite{CHVY}, and is fundamental for these type of free boundary problems.

We should note that our uniqueness result does not complete the full characterization of the long time asymptotic of the fractional thin film equation. Our methods only cover the range $s\in(0,1/2)$, and require a Lipschitz regularity assumption, which is not currently known. Moreover, the convergence result of Segatti and Vazquez \cite{SegattiVazquez}*{Theorem 5.9} requires an extra qualitative condition on the integrability of the gradient of the solution, which is also not currently known.

Lastly, we note that the symmetrization methods in Theorem \ref{thm:uniqueness} hold for any equation that arises as a positive mass constrained critical point of energies that are a combination of:
\begin{itemize}
\item isotropic local first order semi-norms
$$
\int_{\R^n} G(|\nabla v|)\;dx
\quad \hbox{with $G:[0,\infty]\to \R$ convex}
$$
\item isotropic nonlocal semi-norms 
$$
\int_{\R^n}\int_{\R^n} \frac{|v(x)-v(y)|^p}{K(|x-y|)}\;dxdy
\quad \hbox{with $K:[0,\infty]\to[0,\infty]$ increasing}
$$
\item local functionals 
$$
\int_{\R^n}F(v)\;dx \quad \hbox{with $F:[0,\infty]\to\R$}
$$
\item isotropic interaction energies
$$
\frac{1}{2}\int_{\R^n}\int_{\R^n} W(|x-y|)v(x)v(y)\;dxdy
\quad \hbox{with $W:[0,\infty]\to \R$ increasing}
$$
\item radial potential functionals
$$
\int_{\R^n} U(|x|)v(x)\;dx \quad \hbox{with $U:[0,\infty]\to\R$ increasing}.
$$
\end{itemize}
The uniqueness strategy presented here is a bit more touchy and only holds for a strict subset of these equations.

\subsection{Organization of the paper}

The rest of the paper is organized as follows. 
Section \ref{sec:prelim} contains background and preliminary results on continuous Steiner symmetrizations. 
Theorem \ref{thm:main2} is proved in Section \ref{sec:Main Thm}. 
An explicit version of Theorem \ref{thm:main2} for good functions is discussed in Section \ref{sec:good}. Section \ref{sec:interp} presents the height function interpolation and proves Theorem~\ref{thm:nonlocal-convexity}.
Finally, Section \ref{sec:uniqueness} establishes properties of truncated continuous Steiner symmetrizations which are then used to prove Theorem \ref{thm:uniqueness}.

\section{Continuous Steiner symmetrization}\label{sec:prelim}

In this section, we present background and preliminaries on continuous Steiner symmetrizations with constant speed in the direction $e \in \mathbb{S}^{n-1}$. 
For simplicity in the presentation, we will assume that $e = e_n$ 
and denote $x = (x',x_n) \in \R^{n-1} \times \R$.

Given $x' \in \R^{n-1}$, we denote the section of an open subset $U \subset \R^n$ in the direction $e_n$ by
\[
U_{x'} = \{ x_n \in \R : (x',x_n) \in U\}.
\]
The Steiner symmetrization of $U$ with respect to the direction $e_n$ is defined by
\[
S(U) = \{ x = (x',x_n) \in \R^n : x_n \in U_{x'}^*\}
\]
where $U_{x'}^* = \{x_n \in \R : |x_n|< |U_{x'}|/2\}$ is the symmetric rearrangement of $U_{x'}$ in $\R$. 
Note that $|U_{x'}^*| = |U_{x'}|$.
To define $Sf$ for a nonnegative function $f \in L^1(\R^n)\cap C(\R^n)$, we denote the $h>0$ level sets of $f$ in the direction of $e_n$ by
\[
U_{x'}^h = \{ x_n \in \R : f(x',x_n) >h\}.
\]
Then, the Steiner symmetrization $Sf$ of $f$ in the direction $e_n$ 
is given by
\[
Sf(x)  = \int_0^{\infty} \chi_{S(U_{x'}^h)}(x_n) \, dh.
\]

\begin{defn}\label{defn:cont set}
The continuous Steiner symmetrization of an open set $U \subset \R$ is denoted by $M^{\tau}(U)$, $\tau \geq 0$, and defined as follows.
\begin{enumerate}
\item \emph{Intervals}. 
If $U = (y_-, y_+)$, then $M^{\tau}(U) = (y_-^\tau, y_+^{\tau})$ where
\[
\begin{cases}
y_-^\tau &= y_- - \tau  \sgn(y_++y_-)\\
y_+^\tau &= y_+ -\tau  \sgn(y_++y_-)
\end{cases}
\quad \hbox{for}~0 < \tau \leq \frac{\abs{y_++y_-}}{2}
\]
and 
\[
\begin{cases}
y_-^\tau &= - (y_+-y_-)/{2}\\
y_+^\tau &= (y_+-y_-)/{2}
\end{cases}
\quad \hbox{for}~\tau > \frac{\abs{y_++y_-}}{2}.
\]

\item \emph{Finite union of intervals}.
If $U = \bigcup_{i=1}^m I_i$, $m \in \N$, where $I_i$ are disjoint, open intervals, then
\[
M^{\tau}(U) = \bigcup_{i=1}^m M^{\tau}(I_i) \quad \hbox{for}~0 \leq \tau < \tau_1
\]
where $\tau_1$ is the first time that two intervals $M^{\tau}(I_i)$ touch. 
At $\tau_1$, we merge the two intervals and start again.
\item \emph{Countably infinite union of intervals}.
If $U = \bigcup_{i=1}^\infty I_i$ where $I_i$ are disjoint, open intervals, then
\[
M^{\tau}(U) = \bigcup_{i=1}^\infty M^{\tau}(U_i) \quad \hbox{where}~U_m = \bigcup_{i=1}^m I_i,~m \in \N.
\]
\end{enumerate}
\end{defn}

\begin{defn}\label{defn:cont-symm}
The continuous Steiner symmetrization of a nonnegative function $f \in L^1(\R^n)\cap C(\R^n)$ in the direction $e_n$ is denoted by $f^\tau $, $\tau \geq 0$, and defined as
\[
f^\tau (x) = \int_0^{\infty} \chi_{M^\tau(U_{x'}^h)}(x_n) \, dh \quad \hbox{for}~x = (x',x_n) \in \R^{n}.
\]
\end{defn}

By definition, $f^\tau$ interpolates continuously between $f = f^0$ and $Sf = f^\infty$, see Figure \ref{fig:symmetrization}. 
As a consequence of the layer-cake representation, the continuous Steiner symmetrization of $f$ preserves the $L^p$ norm, see \cite{CHVY}*{Lemma 2.14},
\begin{equation}\label{eq:preserves Lp}
\|f\|_{L^p(\R^n)} = \|f^\tau\|_{L^p(\R^n)},
\quad 1 \leq p \leq \infty.
\end{equation}
We also have the semigroup property presented in the next lemma. 

\begin{lem}[Lemma 2.1 in \cite{CHVY}]\label{lem:semigroup}
The collection of operators $(M^\tau)_{\tau \geq 0}$ satisfies the semigroup property. That is, for each $\tau_1, \tau_2\geq 0$ and any open set $U \subset \R$, 
\[
M^{\tau_1}(M^{\tau_2}(U)) = M^{\tau_1 + \tau_2}(U). 
\]
Consequently, $f^\tau$ satisfies the semigroup property:  $(f^{\tau_1})^{\tau_2} = f^{\tau_1+\tau_2}$ for $\tau_1, \tau_2 \geq 0$. 
\end{lem}

A priori, one could define $M^\tau$ with any sufficiently smooth speed $V = V(y,h): \R \times \R_+ \to \R$ by replacing $\sgn(y)$ with $V(y,h)$ in Definition \ref{defn:cont set}(1). 
With a different speed however, $(M^\tau)_{\tau \geq 0}$ will not necessarily satisfy the desired semigroup property. 
Instead, one should replace Definition \ref{defn:cont set}(1) with the following ODE
\begin{equation}\label{eq:ode}
\begin{cases}
\frac{d}{d\tau} [y^{\tau}_{\pm}] = -V(y_+^\tau+y_-^\tau,h) & \tau > 0 \\
y_{\pm}^\tau  = y_{\pm} & \tau = 0.
\end{cases}
\end{equation}
With this modification, $(M^\tau)_{\tau \geq 0}$ satisfies the semi-group property, as long as the level sets remain ordered, see Section~\ref{sec:truncated}. 

\begin{rem}\label{rem:Brockspeed}
The continuous symmetrization considered by Brock \cites{Brock,Brock95} is equivalent to taking $V(y,h)=y$.
\end{rem}

Lastly, we note the following consequence of the semigroup property.

\begin{lem}\label{lem:distance for fixed h}
Let $h>0$ and $\tau_1, \tau_2 \geq 0$. 
For a nonnegative function $f \in L^1(\R^n)\cap C(\R^n)$, we have
\[
\dist(\partial \{f^{\tau_1}> h\}, \partial \{f^{\tau_2}> h\}) \leq |\tau_1-\tau_2|.
\]
\end{lem}

\begin{proof}
If $\tau_2=0$, then it is clear from the definition that
\[
\dist(\partial \{f^{\tau_1}> h\}, \partial \{f> h\}) \leq \tau_1.
\]
The result follows from Lemma \ref{lem:semigroup}.
\end{proof}

\section{On Theorem \ref{thm:main2}}\label{sec:Main Thm}

This section contains the proof of Theorem \ref{thm:main2}. 
We begin by presenting a simplified version of the result in \cite{CHVY}. 

\begin{prop}[Proposition 2.15 in \cite{CHVY}]\label{prop:pastresult}
    Consider $W\in C^1(\R^n)$ an increasing radially symmetric kernel with associated interaction energy
    $$
        \mathcal{I}[f]=\int_{\R^n}\int_{\R^n} f(x)f(y)W(x-y)\, dx\,dy.
    $$
    Assume that $f\in L^1(\R^n)$ is positive and not radially decreasing. Then, there exists constants $\gamma = \gamma(W,f)>0$, $\tau_0 = \tau_0(f)>0$, and a hyperplane $H$, such that
    $$
        \mathcal{I}[f^\tau]\le \mathcal{I}[f]-\gamma \tau \quad \hbox{for all}~0 \leq \tau \leq \tau_0,
    $$
    where $f^\tau$ is the continuous Steiner symmetrization about $H$.
\end{prop}

The original result in \cite{CHVY} allows for more singular kernels, but our prototype kernels $W(x)\approx -|x|^{-n-sp}$ are too singular to directly apply their result. 
In fact, for $\mathcal{I}(f)$ to be well-defined in our setting, one must replace $f(x)f(y)$ by $|f(x)-f(y)|^p$. 
To see this, consider the case $p=2$ and $W(x) = c_{n,s} |x|^{-2s}$ where $c_{n,s}>0$ is the normalizing constant for the fractional Laplacian.
Using the Fourier transform, we can formally write
\[
\mathcal{I}[f] = \int_{\R^n} |\widehat{f}(\xi)|^2 \widehat{W}(\xi) \, d \xi 
\]
but $\widehat{W}(\xi)$ is not defined. On the other hand, using the definition of $(-\Delta)^s$,  
\[
\frac12\int_{\R^n} \int_{\R^n} |f(x) - f(y)|^2 W(x-y) \, dy \, dx
    = \int_{\R^n}f(x)(-\Delta)^s f(x) \, dx 
    = \int_{\R^n}|\widehat{f}(\xi)|^2 |\xi|^{2s}\, d \xi 
\]
which is a well-defined seminorm.

We use an $\varepsilon$-regularization of $W$ for which for Proposition \ref{prop:pastresult} holds. 
For each $0 < \varepsilon \leq 1$, we consider the energy given by
\begin{equation}\label{eq:Fep}
\mathcal{F}_{\varepsilon}^p(f) = \int_{\R^n} \int_{\R^n} |f(x) - f(y)|^pW_{\varepsilon}(x-y) \, dx \, dy, \quad\quad W_{\varepsilon}(x):= \frac{1}{|x|^{n+sp}+\varepsilon}.
\end{equation}
Notice that kernel associated to $\mathcal{F}_{\varepsilon}^p$ is integrable for each fixed $\varepsilon>0$ and that
\[
\lim_{\varepsilon \to 0}\mathcal{F}_{\varepsilon}^p(f) = \sup_{0 < \varepsilon \leq 1} \mathcal{F}_{\varepsilon}^p(f)  = [f]_{W^{s,p}(\R^n)}^p.
\]
Using that $W_{\varepsilon}$ is radially symmetric, the energy $\mathcal{F}^p_\varepsilon(f)$ can be written as
\begin{align*}
\mathcal{F}_{\varepsilon}^p(f) 
    &= \int_{\R^{2n}} \big(|f(x) - f(y)|^p - |f(x)|^p - |f(y)|^p\big) W_{\varepsilon}(x-y) \, dx \, dy\\
    &\quad + \int_{\R^{2n}} \big( |f(x)|^p +|f(y)|^p\big) W_{\varepsilon}(x-y) \, dx \, dy\\
    &= \int_{\R^{2n}} \big(|f(x) - f(y)|^p - |f(x)|^p - |f(y)|^p\big) W_{\varepsilon}(x-y) \, dx \, dy + C_{\varepsilon} \|f\|_{L^p(\R^n)}^p
\end{align*}
where the constant $C_{\varepsilon} = C_{\varepsilon}(s,p)$ satisfies
\begin{equation}\label{eq:Ce}
C_{\varepsilon}  =  2 \int_{\R^n} W_{\varepsilon}(y) \, dy \to \infty \quad \hbox{as}~\varepsilon \to 0^+. 
\end{equation}
For convenience, we define 
\begin{equation}\label{eq:Iep}
\begin{aligned}
\mathcal{I}_{\varepsilon}^p(f) 
    &:= \mathcal{F}_{\varepsilon}^p(f) - C_{\varepsilon}\|f\|_{L^p(\R^n)}^p\\
    &= \int_{\R^{2n}} \big(|f(x) - f(y)|^p - |f(x)|^p - |f(y)|^p\big) W_{\varepsilon}(x-y) \, dx \, dy.
\end{aligned}
\end{equation}

Consider the continuous Steiner symmetrization $f^{\tau}$ of $f$. 
As a consequence of \eqref{eq:preserves Lp}, 
\begin{align*}
\frac{d}{d\tau}\big[\mathcal{F}_{\varepsilon}^p(f^\tau) \big]
    &= \frac{d}{d\tau} \big[ \mathcal{I}_{\varepsilon}^p(f^\tau)  + C_{\varepsilon} \|f\|_{L^p(\R^n)}^p\big]
    =   \frac{d}{d\tau} \big[\mathcal{I}_{\varepsilon}^p(f^\tau) \big].
\end{align*}
In the special case of $p=2$, the integrand in $\mathcal{I}_\varepsilon^2(f)$ simplifies nicely, and we get
\begin{align*}
\frac{d}{d\tau}\big[\mathcal{F}_{\varepsilon}^2(f^\tau) \big]
    &= \frac{d}{d\tau} \bigg[-2 \int_{\R^{2n}} f(x)f(y) W_{\varepsilon}(x-y) \, dx \, dy \bigg]
    =- 2\frac{d}{d\tau} \langle f^\tau, W_{\varepsilon}* f^\tau \rangle_{L^2(\R^n)}.
\end{align*}
Since $\tilde{W}_\eps:=-2W_{\varepsilon} \in C^1(\R^n)$ is symmetric and increasing along its rays, we apply Proposition \ref{prop:pastresult} to find constants $\gamma_{\varepsilon}, \tau_0>0$ such that
\[
 \langle f^\tau, \tilde{W}_\eps* f^\tau \rangle_{L^2(\R^n)} \leq \langle f, \tilde{W}_\eps* f \rangle_{L^2(\R^n)} - \gamma_{\varepsilon} \tau \quad \hbox{for all}~0 \leq \tau \leq \tau_0. 
\]
We will show that $\gamma_{\varepsilon}$ can be bounded uniformly from below in $0 < \varepsilon \leq 1$ and also that we can handle all $1 < p < \infty$. 
More precisely, we prove the following result.

\begin{prop}\label{prop:e-main2}
Let $0 < s < 1$, $1< p < \infty$, and $f \in L^1(\R^n)\cap C(\R^n)$ be nonnegative. Assume that, up to translation or rotation, the nonlinear center of mass is at the origin
\begin{equation}\label{eq:centerofmass}
    \int_{\R^n}\tan^{-1}(x_n)f(x)\;dx=0
\end{equation}
and that $f$ is not symmetric decreasing across the plane $\{x_n=0\}$. Then, there are constants $\gamma = \gamma(n,s,p,f), \tau_0 = \tau_0(f)>0$, independent of $\varepsilon$, such that
\[
\mathcal{I}_{\varepsilon}^p(f^\tau)
 	\leq \mathcal{I}_{\varepsilon}^p(f)- \gamma \tau \quad \hbox{for all}~ 0\leq \tau \leq \tau_0.
\]
\end{prop}
\begin{rem}
    In the original reference \cite{CHVY}, the condition of the nonlinear center of mass \eqref{eq:centerofmass} is replaced by the hyperplane $\{x_n=0\}$ dividing the mass in half. We impose condition \eqref{eq:centerofmass} to simplify some measure theoretical aspects of the proof. We chose the function $\tan^{-1}(x_n)$, because it is odd, strictly monotone, and bounded, which makes the integral well-defined under the assumption $f\in L^1(\R^n)$.
\end{rem}

Since $f^\tau$ is defined as an integral in terms of the one-dimensional level sets $U_{x'}^h \subset \R$ of $f$, 
roughly speaking, one can reduce the proof of Proposition \ref{prop:e-main2} to a one-dimensional setting.  
For a fixed $\ell > 0$, consider the one-dimensional kernel 
\[
K_{\varepsilon}(r) 
	=K_{\varepsilon,\ell}(r) 
	=\frac{1}{(\ell^2 +r^2)^{(n+sp)/2} + \varepsilon}
\quad \hbox{for}~0 < \varepsilon \leq 1,
\]
so that $W_{\varepsilon}(x',x_n) = K_{\varepsilon,|x'|}(x_n)$. 
With the following layer-cake-type representation
\begin{align*}
|f(x) - f(y)|^p
    - |f(x)|^p - |f(y)|^p
&= -p(p-1)\int_0^{\infty} \int_0^{\infty} |h-u|^{p-2} \chi_{U_{x'}^h}(x_n) \chi_{U_{y'}^h}(y_n) \, dh \, du,
\end{align*}
we write \eqref{eq:Iep} as 
\begin{equation}\label{eq:p-interaction}
\begin{aligned}
\mathcal{I}_{\varepsilon}^p(f)
    &= \int_{\R^{2n}} \big(|f(x',x_n) - f(y',y_n)|^p - |f(x',x_n)|^p - |f(y',y_n)|^p\big) K_{\varepsilon,|x'-y'|}(x_n-y_n)  \, dx \, dy\\
    &= -p(p-1)\int_{\R^{2n}}  \int_{\R^2_+}|h-u|^{p-2} \chi_{U_{x'}^h}(x_n) \chi_{U_{y'}^h}(y_n)
        K_{\varepsilon,|x'-y'|}(x_n-y_n)  \, dh \, du \, dx \, dy\\
    &= -p(p-1)\int_{\R^{2(n-1)}}  \int_{\R^2_+} |h-u|^{p-2}\\
    &\quad \int_{\R^{2}}\chi_{U_{x'}^h}(x_n) \chi_{U_{y'}^h}(y_n)
        K_{\varepsilon,|x'-y'|}(x_n-y_n)  \, dx_n \, dy_n \, dh \, du \, dx' \, dy'.
\end{aligned}
\end{equation}
Consequently,
\begin{align*}
\frac{d}{d\tau}\big[\mathcal{I}_{\varepsilon}^p(f^\tau)\big]
    &= -p(p-1)\int_{\R^{2(n-1)}}  \int_{\R^2_+} |h-u|^{p-2} \\
    &\quad \frac{d}{d\tau}\int_{\R^{2}}\big[\chi_{M^\tau(U_{x'}^h)}(x_n) \chi_{M^\tau(U_{y'}^h)}(y_n)
       K_{\varepsilon,|x'-y'|}(x_n-y_n) \big] \, dx_n \, dy_n \, dh \, du \, dx' \, dy'.
\end{align*}
Hence, to establish Proposition \ref{prop:e-main2}, we first study the corresponding  one-dimension problem. 

For open sets $U_1, U_2 \subset \R$, define 
\[
I_{\varepsilon}(\tau) = I_{\varepsilon}[U_1,U_2](\tau):= 
   \int_{\R^{2}}\chi_{M^\tau(U_1)}(x) \chi_{M^\tau(U_2)}(y)
        K_{\varepsilon}(x-y)  \, dx \, dy
\]
where $K_{\varepsilon} = K_{\varepsilon, \ell}$ for a fixed $\ell>0$. 
The main lemma of this section establishes that $I_{\varepsilon}(\tau)$ is strictly increasing in $\tau>0$ 
when $U_1$, $U_2$ are sufficiently separated. 

For a function $g = g(\tau)$, we denote the upper and lower Dini derivatives of $g$ respectively by 
\begin{equation}\label{eq:dini}
\frac{d^+}{d\tau} g(\tau) = \limsup_{\delta \to 0^+} \frac{g(\tau + \delta) - g(\tau)}{\delta}
\quad \hbox{and} 
\quad
\frac{d^-}{d\tau} g(\tau) = 
\limsup_{\delta \to 0^-} \frac{g(\tau + \delta) - g(\tau)}{\delta}.
\end{equation}

\begin{lem}\label{lem:one-d}
Let $U_1, U_2 \subset \R$ be open sets with finite measure. Then 
\begin{equation}\label{eq:1st part}
\frac{d^+}{d\tau} I_{\varepsilon}(\tau) \geq 0 \quad \hbox{for all}~\tau \geq 0.
\end{equation}
If in addition, there exist $0 < a < 1$ and $R> \max \{ \abs{U_1}, \abs{U_2}\}$ such that 
$\abs{U_1 \cap (\abs{U_1}/2,R)} >a$ and 
$\abs{U_2 \cap (-R,-\abs{U_2}/2)} >a$,
then
\begin{equation}\label{eq:2nd part}
\frac{d^+}{d\tau} I_{\varepsilon}(\tau) \geq \frac{1}{128} c a^3 >0 \quad \hbox{for all}~0 \leq \tau \leq \frac{a}{4}
\end{equation}
where 
\[
c = \min \{\abs{K_{1}'(r)} : r \in [a/4, 4R]\}.
\]
\end{lem}

To prove Lemma \ref{lem:one-d}, we apply Propositions 2.16 and 2.17 in \cite{CHVY} to $I_{\varepsilon}$ and show that the upper Dini derivative of $I_{\varepsilon}$ can be uniformly bounded below. 
For the sake of the reader, we first provide a brief, formal argument in the simplest setting. 
Indeed, if $U_i = [c_i-r_i,c_i+r_i]$, $i=1,2$, then we can write
\begin{align}
I_{\varepsilon}(\tau)
	&= \int_{-r_1+c_1 - \tau \sgn(c_1)}^{r_1+c_1 - \tau \sgn(c_1)} \int_{-r_2+c_2 - \tau \sgn(c_2)}^{r_2+c_2 - \tau \sgn(c_2)} K_{\varepsilon}(x-y) \, dy \, dx.
\end{align}
By the semigroup property, it is enough to take the derivative at $\tau=0$ and estimate
\begin{align}
\frac{d^+}{d\tau}I_{\varepsilon}(0)
	&=(\sgn(c_2) - \sgn(c_1)) \int_{-r_1}^{r_1} \int_{-r_2+c_2-c_1}^{r_2+c_2-c_1} K_{\varepsilon}'(x-y)  \, dy \, dx.
\end{align}
If $c_2 > c_1$, then $Q= [-r_1,r_1] \times [-r_2+c_2-c_1,r_2+c_2-c_1]$ is a rectangle in the $xy$ plane centered across $\{x=0\}$.  
Since $K_{\varepsilon}$ is increasing in $\{y>0\}$ and decreasing in $\{y<0\}$, one can show that $\frac{d^+}{d\tau}I_{\varepsilon}(0)>0$.
The more refined lower bound is roughly controlled by the size of the excess strip $[-r_1,r_1]\times[r_2+ (c_2-c_1)/2, r_2+c_2-c_1]$ and by $K_{\varepsilon}'$.

\begin{proof}[Proof of Lemma \ref{lem:one-d}]
First consider when $U_i = [c_i- r_i, c_i+r_i]$, $i=1,2$, are intervals. 
Then, following the proof of \cite{CHVY}*{Lemma 2.16}, we can show that
\[
\frac{d^+}{d\tau} I_{\varepsilon}(0) \geq d_{\varepsilon} \min\{r_1,r_2\} \abs{c_2-c_1}
\]
where
\[
d_{\varepsilon} = \min\left\{\abs{K_{\varepsilon}'(r)}: r \in [\abs{c_2-c_1}/{2}, r_1+r_2 + \abs{c_2-c_1}]\right\}.
\]
Since
\[
K_{\varepsilon}'(r)
    = - \frac{n+sp}{((\ell^2 + r^2)^{(n+sp)/2} +\varepsilon)^2}(\ell^2+r^2)^{(n+sp)/2-1} r
\]
for all $0 < \varepsilon \leq 1$, we have
\begin{equation}\label{eq:K-eps-prime}
\abs{K_{\varepsilon}'(r)}
	\geq \abs{K_1'(r)}.
\end{equation}
Hence, $d_{\varepsilon} \geq d_1$ and
\[
\frac{d^+}{d\tau} I_{\varepsilon}(0) \geq d_{1} \min\{r_1,r_2\} \abs{c_2-c_1} \quad \hbox{for all}~0 < \varepsilon \leq 1. 
\]

With this, we can follow the proof of \cite{CHVY}*{Lemma 2.17} for $U_1, U_2$ finite open sets to show 
\[
\frac{d^+}{d\tau} I_{\varepsilon}(\tau) \geq \frac{1}{128} c_{\varepsilon} a^3 >0 \quad \hbox{for all}~0 \leq \tau \leq \frac{a}{4}
\]
where 
\[
c_{\varepsilon} = \min \{\abs{K_{\varepsilon}'(r)} : r \in [a/4, 4R]\}
	\geq \min \{\abs{K_{1}'(r)} : r \in [a/4, 4R]\} = c. 
\]
\end{proof}

Before proceeding with the proof of Proposition \ref{prop:e-main2}, we will need the following technical lemma for the case $p \not= 2$. 

\begin{lem}\label{lem:h-separation}
Under the assumptions of Proposition~\ref{prop:e-main2},
for $a>0$ small and $R>0$ large, define the sets $B_+^a,B_-^a$ by
\begin{equation}\label{eq:B-sets}
\begin{aligned}
B_+^a &=  \{(x',h) \in \R^{n-1} \times (0,\infty) : |U_{x'}^h \cap (|U_{x'}^h|/2,R)|>a~\hbox{and}~|x'|, \, h \leq R \}\\
B_-^a &=\{(x',h) \in \R^{n-1} \times (0,\infty) : |U_{x'}^h \cap (-R, -|U_{x'}^h|/2)|>a~\hbox{and}~|x'|,\,h\leq R \}.
\end{aligned}
\end{equation}
If $B_+^a$ has positive measure, then there are heights $0<h_1<h_2<\infty$ such that
both
\[
B_+^a \cap \{ h<h_1\} \quad \hbox{and}\quad B_+^a \cap \{h>h_2\}
\]
have positive measure. Similarly for $B_-^a$. 
\end{lem}

\begin{proof}
Consider a density point $(x',h)\in B_+^a$ and the rectangles
$$
\operatorname{Rec}_\delta =\{(y',u): |y'|<R,\; |u-h|<\delta\}.
$$
Note that $|\operatorname{Rec}_\delta|=\omega_n \delta R^{n-1}$ where $\omega_n$ volume of the unit ball in $\R^{n-1}$.  
By density, 
$$
|B_+^a\cap \operatorname{Rec}_\delta|>0 \quad \hbox{for every}~\delta>0.
$$
Let $\delta\ll 1$ be small enough to guarantee 
$$
|B_+^a\cap (\operatorname{Rec}_\delta)^c|>0.
$$
For such a small $\delta$, it follows that
$$
\min(|B_+^a\cap \operatorname{Rec}_{\delta/2}|,|B_+^a\cap (\operatorname{Rec}_\delta)^c|)>0.
$$
The lemma holds by choosing $h_1=h'-\delta$, $h_2=h'-\delta/2$ or  $h_1=h'+\delta/2$, $h_2=h'+\delta$.
\end{proof}

\begin{proof}[Proof of Proposition \ref{prop:e-main2}]
The proof follows along the same lines as the proof of \cite{CHVY}*{Proposition 2.15}. 
We will sketch the idea in order to showcase where we need Lemma \ref{lem:h-separation} and where we apply the estimates in Lemma \ref{lem:one-d} which we have already established to be independent of $0 < \varepsilon \leq 1$.

First, since $f$ is not symmetric across $H = \{x_n =0\}$,  
there exist $a>0$ small and $R>0$ large enough to guarantee that at least one of the sets $B_+^a, B_-^a$ defined in \eqref{eq:B-sets} has positive measure. Due to the nonlinear center of mass condition \eqref{eq:centerofmass}, we know that both of them need to have positive measure.

As a consequence of Lemma \ref{lem:h-separation}, there exist $0 < h_1 < h_2 < \infty$ and $0 < u_1 < u_2 < \infty$ such that the sets
\[
B_+^a \cap \{ h<h_1\}, \quad B_+^a \cap \{ h>h_2\}, \quad B_-^a \cap \{ u < u_1\}, \quad
B_-^a \cap \{ u > u_2\}
\]
all have positive measure. Without loss of generality, assume that $u_1 < h_2$. Otherwise $h_1<u_2$ and the proof is analogous. 

Next, we use  \eqref{eq:p-interaction} and the definitions of $f^\tau$ and $I_{\varepsilon}^p[U_{x'}^h, U_{y'}^u] (\tau)$ to write
\[
\mathcal{I}_{\varepsilon}^p(f^\tau)
    = -p(p-1) \int_{\R^{2(n-1)}}\int_{\R_+^2} |h-u|^{p-2}I_{\varepsilon}^p[U_{x'}^h, U_{y'}^u] (\tau) \, dh \, du \, dx' \, dy'.
\]
Using \eqref{eq:1st part}, we can estimate
\begin{align*}
- \frac{d^+}{d\tau}[\mathcal{I}_{\varepsilon}^p(f^\tau)]
    &\geq p(p-1) \int_{B_-^a \cap \{ u < u_1\}} \int_{B_+^a \cap \{h>h_2\}}  |h-u|^{p-2}\frac{d}{d\tau}[I_{\varepsilon}^p[U_{x'}^h, U_{y'}^u] (\tau)] \, dh \, dx'\, du  \, dy'\\
     &\geq m_{p,f} \int_{B_-^a \cap \{ u < u_1\}} \int_{B_+^a \cap \{h>h_2\}} \frac{d}{d\tau}[I_{\varepsilon}^p[U_{x'}^h, U_{y'}^u] (\tau)] \, dh \, dx'\, du  \, dy'
\end{align*}
where
\[
m_{p,f} := p(p-1)
\begin{cases}
|h_2-u_1|^{p-2} & \hbox{if}~p>2 \\
1 & \hbox{if}~p=2\\
(2R)^{p-2} & \hbox{if}~1<p<2. 
\end{cases}
\]
Applying now \eqref{eq:2nd part} and following the proof of \cite{CHVY}*{Proposition 2.15}, we obtain
\begin{align*}
- \frac{d^+}{d\tau}\mathcal{I}_{\varepsilon}^p(f^\tau)\bigg|_{\tau=0}
    &\geq \frac{m_{p,f}}{6000}  |B_+^a \cap  \{h>h_2\}||B_-^a \cap \{ u < u_1\}| \min_{r \in [a/4,4R]} |W'_1(r)| a^4 
    >0
\end{align*}
for all $0 \leq \tau \leq a/4$ where, with an abuse of notation, $W_1(r)$ is such that $W_1(|x|):= W_1(x)$. 
\end{proof}

We conclude this section with the proofs of Theorem \ref{thm:main2} and Corollary \ref{cor:Lip}.

\begin{proof}[Proof of Theorem \ref{thm:main2}]
Without loss of generality, assume that $H = \{x_n = 0\}$ coincides with the nonlinear center of mass condition and that $f$ is not symmetric decreasing across $H$.
For $0 < \varepsilon \leq 1$, let $\mathcal{F}_{\varepsilon}^p$ and $\mathcal{I}_{\varepsilon}^p$ be as in \eqref{eq:Fep} and  \eqref{eq:Iep}, respectively.
Using \eqref{eq:preserves Lp}, we have 
\begin{align*}
\mathcal{F}_\varepsilon^p(f^\tau) 
	&= C_{\varepsilon} \|f\|_{L^p(\R^n)}^p + \mathcal{I}_{\varepsilon}^p(f^\tau) \\
	&=\mathcal{F}_\varepsilon^p(f) + \mathcal{I}_{\varepsilon}^p(f^\tau)-\mathcal{I}_{\varepsilon}^p(f).
\end{align*}
By Proposition \ref{prop:e-main2}, there are constants $\gamma = \gamma(n,s,p,f), \tau_0 = \tau_0(f)>0$, independent of $\varepsilon$, such that
\[
\mathcal{F}_\varepsilon^p(f^\tau) 
	\leq \mathcal{F}_\varepsilon^p(f)  - \gamma \tau \quad \hbox{for all}~0 \leq \tau \leq \tau_0.
\]
The statement follows by taking $\varepsilon \to 0^+$. 
\end{proof}

\begin{rem}\label{rem:gamma-limit}
Notice from the proof of Proposition \ref{prop:e-main2} that
\[
\lim_{ s \to1^-}\min_{r \in [a/4,4R]} |W'_1(r)|
    =\min_{r \in [a/4,4R]} \frac{(n+p)r^{n+p-1}}{(r^{n+p}+1)^2} >0
\]
and also 
\[
\lim_{ s \to 0^+}\min_{r \in [a/4,4R]} |W'_1(r)|
    =\min_{r \in [a/4,4R]} \frac{nr^{n-1}}{(r^{n}+1)^2} >0.
\]
Therefore, we have
\[
\lim_{ s \to1^-} \gamma(n,s,p,f)>0 \quad \hbox{and} \quad \lim_{ s \to 0^+} \gamma(n,s,p,f)>0.
\]
and we have $s(1-s) \gamma \to 0$ as $s \to 0^+, 1^-$. 
After multiplying both sides of \eqref{eq:main estimate} by $s(1-s)$ and taking the limit as $s \to 1^-$, we obtain Corollary \ref{cor:Lip}. 
Note that if we instead take $s \to 0^+$, we do not contradict \eqref{eq:preserves Lp}. 
\end{rem}

\section{Explicit representations for good functions}\label{sec:good}

Here, we define and establish preliminary results for good functions, then we prove an explicit version of Theorem \ref{thm:main2} for good functions. 

\subsection{Good functions and local energies}

We begin by presenting the definition of \textit{good functions} and highlight their uses, which can be found in Brock's work \cites{Brock,Brock95}. 
As Definition \ref{defn:cont set} is broken down into cases for which the open set $U \subset \R$ is either an interval, a finite union of intervals, or an infinite union of intervals. 
Good functions are those functions whose sections $U_{x'}^h$ are a finite union of intervals which allows for explicit computation of both the fractional energy for $f^\tau$ and its derivative in $\tau$. 

\begin{defn}\label{defn:good}
A nonnegative, piecewise smooth function $f = f(x',x_n)$, $x' \in \R^{n-1}$, $x_n \in \R$, with compact support is called a \emph{good function} if 
\begin{enumerate}
\item for every $x' \in \R^{n-1}$ and every $h>0$ except a finite set, the equation $f(x',x_n) = h$ has exactly $2m$ solutions, denoted by $x_n = x_n^{\ell}(x,h)$, satisfying $x_n^{\ell}< x_n^{\ell+1},$ $\ell=1,\dots,2m$ where $m = m(x',h) < \infty$, and
\item $\displaystyle{\inf\left\{\abs{\frac{\partial f}{\partial x_n}(x',x_n)}: x' \in \R^{n-1},~x_n \in \R,~\hbox{and}~\frac{\partial f}{\partial x_n}(x',x_n)~\hbox{exists}\right\}>0}$.
\end{enumerate}
\end{defn}

The functions illustrated in Figure \ref{fig:symmetrization} and below in Figure \ref{fig:local-ex} are good functions. In general, one might think of good functions as a collection of peaks, creating a mountain range. 

\begin{notation}\label{notation:xs}
We denote the solutions $x_n$ to $f(x',x_n) = h$ using subscript notation $x_{2k-1}< x_{2k}$ for $k=1,\dots, m= m_h$. 
(This is not to be confused with the subscripts in $x' = (x_1,\dots, x_{n-1})$.)
In the case of $m=1$ or when considering an arbitrary interval $(x_{2k-1}, x_{2k})$, we will commonly adopt the notation $x_+:= x_{2k}$ and $x_-:= x_{2k-1}$. 
We will also denote solutions $y_n$ to $f(y',y_n) = u$ by $y_{2\ell-1}< y_{2\ell}$ for $\ell = 1,\dots, m = m_u$. 
\end{notation}

\begin{rem}\label{rem:good for symmetric}
If $f$ is a good function that is symmetric and decreasing across $\{x_n=0\}$, then it must be that $m=1$ and $x_-=-x_+$ for all $0 < h < \|f\|_{L^{\infty}(\R^n)}$. 
\end{rem}

Just as finite union of intervals can be used to approximate open sets in $\R$, good functions can be used to approximate Sobolev functions.  

\begin{lem}[See \cite{Brock}]
\label{lem:density}
\mbox{}
\begin{enumerate}
\item Good functions are dense in $W^{1,p}_+(\R^n)$ for every $1 \leq p < \infty$.
\item If $f$ is a good function, then $f^{\tau}$ is a good function for $0 \leq \tau \leq \infty$.
\end{enumerate}
\end{lem}

Good functions are a powerful tool for continuous Steiner symmetrizations as they allow us to take the $\tau$-derivative directly and expose a quantification of asymmetry. 
Even in the local setting (see Corollary \ref{cor:Lip}), we can explicitly estimate 
the derivative in $\tau$ of $\|f^\tau\|_{W^{1,p}(\R^n)}^p$ when $f$ is a good function.
This is in contrast to Brock's original approach in \cite{Brock} which relies on convexity to estimate the difference in norms of $f$ and $f^\tau$. 

We present the following discussion
for the interested reader to showcase the convenience of using good functions in computations. 

Since the sign function is not differentiable at the origin, fix $\varepsilon>0$ and 
let $\delta_\varepsilon$ be an $\varepsilon$-regularization of the usual Dirac delta
$$
\delta_\varepsilon(x)=\begin{cases}
    \frac{1}{2\varepsilon}&\mbox{if $|x|<\varepsilon$}\\
    0&\mbox{if $|x|>\varepsilon$},
\end{cases}
$$
and we consider $V_\varepsilon$ in \eqref{eq:ode} such that $V_\eps' = 2 \delta_\eps$. 
Let $M^{\tau,\eps}(U)$ denote the continuous Steiner symmetrization of an open set $U \subset \R$ with speed $V_\eps$ and 
$f^{\tau,\eps}$ denote the corresponding continuous Steiner symmetrization of $f$.
As the regularization parameter $\eps \to 0^+$, we recover the original rearrangement:
\begin{lem}
Let $f:\R^n \to \R$ be a good function.  
For any continuous function $g:\R^n \to \R$, it holds that
    $$
        \lim_{\varepsilon\to 0^+}\int_{\R^n} f^{\tau,\varepsilon}g\;dx=\int_{\R^n} f^{\tau}g\;dx.
    $$
\end{lem}

\begin{proof}
\noindent \textit{Step 1.} We consider $U=\bigcup_{i=1}^r(x_{2i-1},x_{2i})$ an open set which is the union of $r\in\N$ intervals. Then, we can show that the bound
\begin{equation}\label{eq:eps-claim}
|M^{\tau,\eps}(U) 
\triangle M^{\tau}(U)| \leq (r+3) \eps \quad \hbox{for all}~\tau\ge 0.
\end{equation}

\hspace{0.3cm}

\noindent\textit{Proof of Step 1.} We start by perturbing the set $U$ by considering $U^\eps=U\cup(-\eps,\eps)=\bigcup_{j=1}^s(x^\eps_{2i-1},x^\eps_{2i})$, with $s\le r+1$. As we added the set $(-\eps,\eps)$, there can exist utmost one interval satisfying $I_{i_0}=(x^\eps_{2i_0-1},x^\eps_{2i_0})$ satisfying $|x^\eps_{2i_0-1}+x^\eps_{2i_0}|\le \eps$. 
We expand this interval to center it, namely we consider the new interval $\tilde{I}_{i_0}=(\tilde{x}^\eps_{2i_0-1},\tilde{x}^\eps_{2i_0})=I_{i_0}\cup P$ such that $|P|\le \eps$ and $|\tilde{x}^\eps_{2i_0-1}-\tilde{x}^\eps_{2i_0}|=0$. If there are any new non-trivial intersections with $\tilde{I}_{i_0}$, we re-label the intervals and repeat the process. This procedure finishes in $k_0$-steps with $1\le k_0\le r+1$, where $r$ is the original number of intervals. Hence, we constructed an open set $S(U^\eps)$, that satisfies $U^\eps \subset S(U^\eps)$ with utmost $r-k_0+2$ disjoint intervals and $|U^\eps \triangle S({U}^\eps)|\le k_0\eps$. The advantage of the set $S(U^\eps)$ is that the rearrangements coincide
$$
M^{\tau,\eps}(S(U^\eps))=M^{\tau}(S(U^\eps))\qquad \hbox{for all}~
\tau\le \tau_1,
$$
where $\tau_1$ is larger than the the first time when two intervals meet. At $\tau_1$, when there exist an interval $I_{i_1}=(x^\eps_{2i_1-1},x^\eps_{2i_1})$ satisfying $|x^\eps_{2i_1-1}+x^\eps_{2i_1}|\le \eps$, we repeat the procedure of enlargement and centering from before. This process ends in $k_1$ steps with $1\le k_1\le n-k_0+1$ steps, and producing a set new set $S(M^{\tau_1}(S(U^\eps))$ that satisfies $M^{\tau_1}(S(U^\eps)) \subset S(M^{\tau_1}(S(U^\eps)))$ with utmost $r-k_0-k_1+2$ disjoint intervals and $|M^{\tau_1}(S(U^\eps)) \triangle S(M^{\tau_1}(S(U^\eps)))|\le k_1\eps$. Again, the set $S(M^{\tau_1}(S(U^\eps)))$ satisfies that the re-arrangements coincide
$$
M^{\tau,\eps}(S(M^{\tau_1}(S(U^\eps)))=M^{\tau}(S(M^{\tau_1}(S(U^\eps)))\qquad \hbox{for all}~
\tau\le \tau_2,
$$
where $\tau_2$ is larger than the time when two intervals meet. Continuing this process inductively, we can produce a \textit{discontinuous} family of open set $\{U_\tau^\eps\}_{\tau>0}$. Using that the rearragements preserve containment, we can obtain 
$$
M^\tau(U)\cup M^{\tau,\eps}(U)\subset U_\tau^\eps\qquad
\hbox{for all}~\tau>0.
$$ 
This process adds utmost $r+3$
intervals of size $\varepsilon$. 
Hence,
$$
|M^\tau(U)\triangle M^{\tau,\eps}(U)|\le | M^\tau(U)\triangle U_\tau^\eps|\le (r+3)\eps.
$$

\hspace{0.3cm}

\noindent\textit{Step 2.} We show that for any good function $f_r$ 
satisfying that $U^h_{x'}$ has utmost $r\in\N$ intervals for every $x'\in\R^{n-1}$ and every $h\ge 0$, we have for any $g\in C(\R^n)$ 
that
    $$
        \lim_{\varepsilon\to 0^+}\int_{\R^n} f_r^{\tau,\varepsilon}g\;dx=\int_{\R^n} f_r^{\tau}g\;dx.
    $$

\hspace{0.3cm}

\noindent\textit{Proof of Step 2.}
Using the representation,
\begin{align*}
\left|\int_{\R^n} g(x) f^{\tau,\varepsilon}(x)\;dx-\int_{\R^n} g(x) f^{\tau}(x)\;dx\right|
   &= \left|\int_{\supp f} \left(f^{\tau,\varepsilon}(x)-f^{\tau}(x)\right) g(x)\;dx\right| \\
   &= \left|\int_0^{\|f\|_{\infty}} \int_{\supp f} (\chi_{M^{\tau,\eps}(U_{x'}^h)}-\chi_{M^{\tau}(U_{x'}^h)} )g(x) \, dx dh\right|\\ 
   &\le \|g\|_{L^1(\supp f)}\|f\|_{\infty} \sup_{x',h} |M^{\tau,\eps}(U_{x'}^h)\triangle M^{\tau}(U_{x'}^h)|\\
   &\le \|g\|_{L^1(\supp f)}\|f\|_{\infty} (r+3)\eps\to 0,
\end{align*}
where we have used \textit{Step 1.} for the last bound.

\hspace{0.3cm}

\noindent
\textit{Step 3.}~To conclude the proof, we can approximate any good function $f$ by a sequence $\{f_r\}_{r\in\N}$ 
of good functions such that $f_r$ satisfies the properties of \textit{Step 2}.

\end{proof}

We now present an explicit estimate on the derivative of $\|f^\tau\|_{W^{1,p}(\R^n)}^p$.  
For simplicity, we will only state the case of $n=2$.

\begin{prop}\label{lem:good-local}
Assume that $f = f(x,y)$ is a non-negative good function, then
\begin{align*}
\frac{d}{d\tau}[f^\tau]_{W^{1,p}(\R^2)}^p&\bigg|_{\tau=0}
\le - \liminf_{\varepsilon \to 0^+}\int_{\R}\int_0^{\infty} \sum_{k=1}^{m_h}\delta_{\varepsilon}(y_{2k}+y_{2k-1}) \bigg|\frac{\partial y_{2k}}{\partial h}\bigg|^{-p} \bigg|\frac{\partial y_{2k-1}}{\partial h}\bigg|^{-p}\bigg[\\
    &\quad p\bigg( \bigg(\frac{\partial y_{2k}}{\partial x}\bigg)^2  
        + 1\bigg)^{p/2-1}
        \bigg(\frac{\partial y_{2k}}{\partial x}\bigg) \bigg(\frac{\partial y_{2k}}{\partial x} + \frac{\partial y_{2k-1}}{\partial x}\bigg)
        \bigg|\frac{\partial y_{2k}}{\partial h}\bigg|\bigg|\frac{\partial y_{2k-1}}{\partial h}\bigg|^p\\
    &\quad 
        +p\bigg( \bigg(\frac{\partial y_{2k-1}}{\partial x}\bigg)^2  
        + 1\bigg)^{p/2-1}
        \bigg(\frac{\partial y_{2k-1}}{\partial x}\bigg)  \bigg(\frac{\partial y_{2k}}{\partial x} + \frac{\partial y_{2k-1}}{\partial x}\bigg)
        \bigg|\frac{\partial y_{2k-1}}{\partial h}\bigg|\bigg|\frac{\partial y_{2k}}{\partial h}\bigg|^p
        \\
        &\quad - (p-1) \bigg( \bigg(\frac{\partial y_{2k}}{\partial x}\bigg)^2  
        + 1\bigg)^{p/2} \bigg|\frac{\partial y_{2k-1}}{\partial h}\bigg|^p \bigg(\bigg|\frac{\partial y_{2k}}{\partial h}\bigg| -\bigg| \frac{\partial y_{2k-1}}{\partial h}\bigg|\bigg)\\
         &\quad+ (p-1) \bigg( \bigg(\frac{\partial y_{2k-1}}{\partial x}\bigg)^2  
        + 1\bigg)^{p/2}\bigg|\frac{\partial y_{2k}}{\partial h}\bigg|^p
        \bigg(\bigg|\frac{\partial y_{2k}}{\partial h}\bigg| -\bigg| \frac{\partial y_{2k-1}}{\partial h}\bigg|\bigg)
        \bigg]\, dh \, dx
\end{align*}
where $f(x,y_i(h)) = h$ for $i=1,\dots,2m_h$ and $m_h = m(x,h)$.
\end{prop}

When $p=2$, we can further factor the integrand to readily check the sign of the derivative. 

\begin{cor}\label{cor:local deriv}
Assume that $f = f(x,y)$ is a non-negative good function, then
\begin{align*}
\frac{d}{d\tau}&[f^\tau]_{H^1(\R^2)}^2\bigg|_{\tau=0}\\
    &\le -\liminf_{\varepsilon \to 0^+}\int_{\R}\int_0^{\infty} \sum_{k=1}^{m_h} \delta_{\varepsilon}(y_{2k}+y_{2k-1}) \bigg|\frac{\partial y_{2k}}{\partial h}\bigg|^{-2} \bigg|\frac{\partial y_{2k-1}}{\partial h}\bigg|^{-2}\bigg(\bigg|\frac{\partial y_{2k}}{\partial h}\bigg| + \bigg|\frac{\partial y_{2k-1}}{\partial h}\bigg|\bigg)\\
&\quad \bigg[ \bigg(\bigg|\frac{\partial y_{2k}}{\partial h}\bigg| - \bigg|\frac{\partial y_{2k-1}}{\partial h}\bigg|\bigg)^2 
    + A_k(h) \bigg\langle \frac{\partial y_{2k}}{\partial x}, \frac{\partial y_{2k-1}}{\partial x} \bigg\rangle \cdot \bigg\langle \frac{\partial y_{2k}}{\partial x}, \frac{\partial y_{2k-1}}{\partial x} \bigg\rangle \bigg]
    \leq 0
\end{align*}
where $A_k(h)$ is the positive semidefinite matrix
\[
A_k(h) = \begin{pmatrix}
\big|\frac{\partial y_{2k-1}}{\partial h}\big|^2 & \big|\frac{\partial y_{2k}}{\partial h}\big|\big|\frac{\partial y_{2k-1}}{\partial h}\big|\\[.25em]
\big|\frac{\partial y_{2k}}{\partial h}\big|\big|\frac{\partial y_{2k-1}}{\partial h}\big| & \big|\frac{\partial y_{2k}}{\partial h}\big|^2
\end{pmatrix}.
\]
\end{cor}

\begin{rem}
In the case of $p\not=2$, 
we have checked numerically that the integrand in the expression for the derivative is indeed negative, but due to the nonlinearity, we have not been able to analytically observe the sign as cleanly as in the case of $p=2$. 
\end{rem}

Notice in Corollary \ref{cor:local deriv} that the derivative is strictly negative if and only if
\begin{equation}\label{eq:asymmetric}
\frac{\partial y_{2k}}{\partial h}\not= - \frac{\partial y_{2k-1}}{\partial h}
\end{equation}
on a set of positive measure. 
Indeed, as mentioned after Corollary \ref{cor:Lip}, the strict inequality does not hold in general. 
For instance, the derivative in $\tau$ can be zero when $\supp f$ is not connected, but $f$ is radially decreasing in each connected component.
We illustrate this with a simple example.

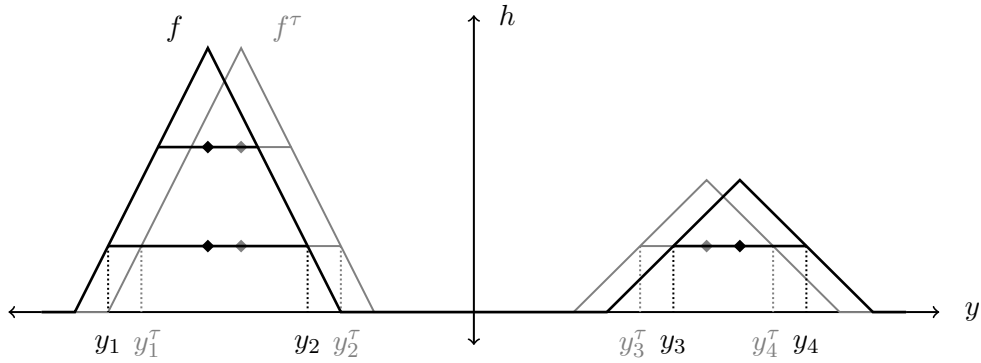
\begin{figure}[htbp]
 \begin{tikzpicture}[scale=1.75, use Hobby shortcut, closed=false]
 \tikzset{bullet/.style={diamond,fill,inner sep=1.25pt}}
\draw[line width=.75pt,<->] (0,-.25)--(0,2.25); 
	\draw (.25,2.25) node {$h$};
\draw[line width=.75pt,<->] (-3.5,0)--(3.5,0); 
	\draw (3.75,0) node {$y$};
\draw[line width=.75pt, gray] (-3.25,0)--(-2.75,0)--(-1.75,2)--(-.75,0) 
		--(.75,0)--(1.75,1)--(2.75,0)--(3.25,0); 
	\draw[gray] (-1.4,2.15) node { $f^\tau$};
\draw[line width=.75pt, gray] (-2.5,.5)--(-1,.5);
	\draw[gray] (-1.75,.5) node[bullet] {}; 
\draw[line width=.75pt,gray] (1.25,.5)--(2.25,.5); 
	\draw[gray] (1.75,.5) node[bullet] {};
\draw[line width=.75pt,gray] (-2.13,1.25)--(-1.37,1.25); 
	\draw[gray] (-1.75,1.25) node[bullet] {}; 
\draw[line width=.75pt, densely dotted, gray] (-2.5,.5)--(-2.5,0);
	\draw[gray] (-2.45,-.24) node { $y_1^\tau$};
\draw[line width=.75pt, densely dotted,gray] (-1,.5)--(-1,0);
	\draw[gray] (-.95,-.24) node { $y_2^\tau$};
\draw[line width=.75pt, densely dotted,gray] (1.25,.5)--(1.25,0);
	\draw[gray] (1.2,-.24) node { $y_3^\tau$};
\draw[line width=.75pt, densely dotted,gray] (2.25,.5)--(2.25,0);
	\draw[gray] (2.2,-.24) node { $y_4^\tau$};
\draw[line width=1pt] (-3.25,0)--(-3,0)--(-2,2)--(-1,0) 
		--(1,0)--(2,1)--(3,0)--(3.25,0); 
	\draw (-2.25,2.15) node { $f$};
\draw[line width=1pt] (-2.75,.5)--(-1.25,.5); 
	\draw (-2,.5) node[bullet] {}; 
\draw[line width=1pt] (1.5,.5)--(2.5,.5); 
	\draw (2,.5) node[bullet] {};
\draw[line width=1pt] (-2.38,1.25)--(-1.62,1.25); 
	\draw (-2,1.25) node[bullet] {}; 
\draw[line width=.75pt, densely dotted] (-2.75,.5)--(-2.75,0);
	\draw (-2.75,-.25) node { $y_1$};
\draw[line width=.75pt, densely dotted] (-1.25,.5)--(-1.25,0);
	\draw (-1.25,-.25) node { $y_2$};
\draw[line width=.75pt, densely dotted] (1.5,.5)--(1.5,0);
	\draw (1.5,-.25) node { $y_3$};
\draw[line width=.75pt, densely dotted] (2.5,.5)--(2.5,0);
	\draw (2.5,-.25) node { $y_4$};
\end{tikzpicture}
\caption{Graph of $f$ and $f^\tau$ in Example \ref{ex:zero} with $\tau = .25$, $x=0$}
\label{fig:local-ex}
\end{figure}

\begin{ex}\label{ex:zero}
Consider the good function $f:\R^2 \to \R$ given by
\[
f(x,y) = \begin{cases}
    -|(x,y-2)|+1 & \hbox{if}~|(x,y-2)| \leq 1\\
    -2|(x,y+2)|+2 & \hbox{if}~|(x,y+2)| \leq 1\\
    0 & \hbox{otherwise}
\end{cases}
\]
and illustrated in Figure \ref{fig:local-ex} at $x=0$. One can readily check that \emph{equality} holds in \eqref{eq:asymmetric}, so that the integrand in Proposition \ref{lem:good-local} is exactly zero. 
\end{ex}

\begin{proof}[Proof of Proposition~\ref{lem:good-local}]
Begin by writing the energy as
\[
[f^\tau]_{W^{1,p}(\R^2)}^p
    = \int_{\R^2} |\nabla f^\tau|^p \, dy \, dx
    = \int_{\R^2} \bigg( \bigg(\frac{\partial f^\tau}{\partial x}\bigg)^2 + \bigg(\frac{\partial f^\tau}{\partial y}\bigg)^2\bigg)^{p/2}\, dy \, dx.
\]
For each fixed $x \in \R$, we make the change of variables $f(x,y) = h$ using the co-area formula in the variable $y$ to write
\begin{align*}
[f^\tau]_{W^{1,p}(\R^n)}^p
    &= \int_{\R^2} \bigg( \bigg(\frac{\partial f^\tau}{\partial x}\bigg)^2 + \bigg(\frac{\partial f^\tau}{\partial y}\bigg)^2\bigg)^{p/2} \bigg|\frac{\partial f}{\partial y}\bigg|^{-1}\bigg|\frac{\partial f}{\partial y}\bigg| \, dy \, dx\\
   &= \int_{\R}\int_0^{\infty}  \bigg( \bigg(\frac{\partial f^\tau}{\partial x}\bigg)^2 + \bigg(\frac{\partial f^\tau}{\partial y}\bigg)^2\bigg)^{p/2} \bigg|\frac{\partial f}{\partial y}\bigg|^{-1} \bigg|_{\{f(x,y) = h\}} \, dh \, dx.
\end{align*}
Since $f$ is a good function, for each $h>0$ except a finite set, there are at most $2m_h$ solutions to $f(x,y) = h$. We denotes these by $y_i = y_i(x,h)$, $i = 1,\dots, 2m_h$. 
In the new variables, we have
\begin{equation}\label{eq:y-derivatives}
\frac{\partial f^\tau}{\partial y}(x, y_+) = \bigg( \frac{\partial y_+^\tau}{\partial h}\bigg)^{-1}<0, \quad
\frac{\partial f^\tau}{\partial y}(x, y_-) = \bigg( \frac{\partial y_-^\tau}{\partial h}\bigg)^{-1}>0
\end{equation}
and
\begin{align*}
\frac{\partial f^\tau}{\partial x}(x, y_+) = \frac{\partial y_+^\tau}{\partial x} \bigg( \frac{\partial y_+^\tau}{\partial h}\bigg)^{-1}, \quad
\frac{\partial f^\tau}{\partial x}(x, y_-) = \frac{\partial y_-^\tau}{\partial x}\bigg( \frac{\partial y_-^\tau}{\partial h}\bigg)^{-1}
\end{align*}
for an arbitrary $y_-< y_+$ (recall Notation \ref{notation:xs}).
Then, we write
\begin{align*}
[f^\tau]_{W^{1,p}(\R^n)}^p
    &= \int_{\R}\int_0^{\infty} \sum_{k=1}^{m_h}\bigg[
    \bigg( \bigg(\frac{\partial y_{2k}^\tau}{\partial x}\bigg)^2 \bigg( \frac{\partial y_{2k}^\tau}{\partial h}\bigg)^{-2} 
        + \bigg( \frac{\partial y_{2k}^\tau}{\partial h}\bigg)^{-2}\bigg)^{p/2} \bigg|\frac{\partial y_{2k}^\tau}{\partial h}\bigg|\\
    &\quad + \bigg( \bigg(\frac{\partial y_{2k-1}^\tau}{\partial x}\bigg)^2 \bigg( \frac{\partial y_{2k-1}^\tau}{\partial h}\bigg)^{-2} 
        + \bigg( \frac{\partial y_{2k-1}^\tau}{\partial h}\bigg)^{-2}\bigg)^{p/2} \bigg|\frac{\partial y_{2k-1}^\tau}{\partial h}\bigg|\bigg]\, dh \, dx.
\end{align*}
For simplicity in the proof, let us assume that $m_h = 1$ or $m_h=0$ for all $h$ except a finite set, and write
\begin{align*}
[f^\tau]_{W^{1,p}(\R^n)}^p
    &= \int_{\R}\int_0^{\infty}\bigg[
    \bigg( \bigg(\frac{\partial y_{+}^\tau}{\partial x}\bigg)^2 \bigg( \frac{\partial y_{+}^\tau}{\partial h}\bigg)^{-2} 
        + \bigg( \frac{\partial y_{+}^\tau}{\partial h}\bigg)^{-2}\bigg)^{p/2} \bigg|\frac{\partial y_{+}^\tau}{\partial h}\bigg|\\
    &\quad + \bigg( \bigg(\frac{\partial y_{-}^\tau}{\partial x}\bigg)^2 \bigg( \frac{\partial y_{-}^\tau}{\partial h}\bigg)^{-2} 
        + \bigg( \frac{\partial y_{-}^\tau}{\partial h}\bigg)^{-2}\bigg)^{p/2} \bigg|\frac{\partial y_{-}^\tau}{\partial h}\bigg|\bigg]\, dh \, dx\\
    &= \int_{\R}\int_0^{\infty}\bigg[
    \bigg( \bigg(\frac{\partial y_{+}^\tau}{\partial x}\bigg)^2  
        + 1\bigg)^{p/2} \bigg|\frac{\partial y_{+}^\tau}{\partial h}\bigg|^{1-p}
        + \bigg( \bigg(\frac{\partial y_{-}^\tau}{\partial x}\bigg)^2  
        + 1\bigg)^{p/2} \bigg|\frac{\partial y_{-}^\tau}{\partial h}\bigg|^{1-p}\bigg]\, dh \, dx.
\end{align*}
We have now arrived at a useful expression for taking the derivative in $\tau$. 
Indeed, let $V_{\varepsilon}$ be an approximation of the speed $V(y) = \sgn(y)$ such that $\delta_{\varepsilon} =V_{\varepsilon}'$.
Then, from \eqref{eq:ode} with speed $V_{\varepsilon}$, 
we obtain
\begin{align*}
\frac{\partial^2 y_{\pm}^\tau}{\partial \tau \partial h} 
    = - \delta_{\varepsilon}(y_++y_-) \bigg(\frac{\partial y_+}{\partial h} + \frac{\partial y_-}{\partial h}\bigg)
\end{align*}
and similarly
\begin{align*}
\frac{\partial^2 y_{\pm}^\tau}{\partial \tau \partial x} 
    = - \delta_{\varepsilon}(y_++y_-) \bigg(\frac{\partial y_+}{\partial x} + \frac{\partial y_-}{\partial x}\bigg).
\end{align*}
With this, we can use the lower-semincontinuity of the $W^{1,p}$ seminorm, to have
\begin{align*}
\frac{d}{d\tau}&[f^\tau]_{W^{1,p}(\R^n)}^p\bigg|_{\tau=0}\le \liminf_{\eps\to 0^+}\frac{d}{d\tau}[f^{\tau,\eps}]_{W^{1,p}(\R^n)}^p\bigg|_{\tau=0}\\
    &\le \liminf_{\varepsilon \to 0^+}\int_{\R}\int_0^{\infty}\bigg[
    -p\bigg( \bigg(\frac{\partial y_{+}}{\partial x}\bigg)^2  
        + 1\bigg)^{p/2-1}
        \bigg(\frac{\partial y_{+}}{\partial x}\bigg) \delta_{\varepsilon}(y_++y_-) \bigg(\frac{\partial y_+}{\partial x} + \frac{\partial y_-}{\partial x}\bigg)
        \bigg|\frac{\partial y_{+}}{\partial h}\bigg|^{1-p}\\
    &\quad 
        -p\bigg( \bigg(\frac{\partial y_{-}}{\partial x}\bigg)^2  
        + 1\bigg)^{p/2-1}
        \bigg(\frac{\partial y_{-}}{\partial x}\bigg) \delta_{\varepsilon}(y_++y_-) \bigg(\frac{\partial y_+}{\partial x} + \frac{\partial y_-}{\partial x}\bigg)
        \bigg|\frac{\partial y_{-}}{\partial h}\bigg|^{1-p}
        \\
        &\quad+ (p-1) \bigg( \bigg(\frac{\partial y_{+}}{\partial x}\bigg)^2  
        + 1\bigg)^{p/2} \bigg|\frac{\partial y_{+}}{\partial h}\bigg|^{-p-1}\bigg(\frac{\partial y_{+}}{\partial h}\bigg)\delta_{\varepsilon}(y_++y_-) \bigg(\frac{\partial y_+}{\partial h} + \frac{\partial y_-}{\partial h}\bigg)\\
         &\quad+ (p-1) \bigg( \bigg(\frac{\partial y_{-}}{\partial x}\bigg)^2  
        + 1\bigg)^{p/2} \bigg|\frac{\partial y_{-}}{\partial h}\bigg|^{-p-1}\bigg(\frac{\partial y_{-}}{\partial h}\bigg)\delta_{\varepsilon}(y_++y_-) \bigg(\frac{\partial y_+}{\partial h} + \frac{\partial y_-}{\partial h}\bigg)
        \bigg]\, dh \, dx.
\end{align*}
Recalling \eqref{eq:y-derivatives} and factoring gives the desired expression.
\end{proof}

\subsection{Explicit derivative computation for good functions}
The objective of this section is to compute the derivative of the nonlocal energy explicitly. With the same notation as in Section \ref{sec:Main Thm}, we consider again $\mathcal{F}_\varepsilon^p$ and $\mathcal{I}_\varepsilon^p$
defined in \eqref{eq:Fep} and
\eqref{eq:Iep}, respectively.
When $f$ is a good function, we can write each $U_{x'}^h$ in \eqref{eq:p-interaction} as a finite union of open intervals.
This provides a useful expression of $\mathcal{I}_\varepsilon^p$ which allows for more explicit computation. 

It will be useful to notate the first and second anti-derivatives of $K(r) = K_{\varepsilon, |x'|}(r)$ in $r$ by
\begin{equation}\label{eq:antideriv}
\bar{K}(r) := \int_0^r K(\xi) \, d\xi,
    \quad \bar{\bar{K}}(r) = \int_0^r \int_0^\rho K(\xi) \, d \xi \, d\rho. 
\end{equation}
We now present the main result of this section.  

\begin{prop} \label{lem:deriv for good}
Let $\varepsilon \geq 0$. 
Assume that $f$ is a non-negative good function. Then, 
\begin{equation}\label{eq:deriv level sets}
\begin{aligned}
\frac{d}{d\tau}\mathcal{F}_{\varepsilon}^p(f^\tau) \bigg|_{\tau=0}
	 &=-p (p-1)\int_{\R^{n-1}} \int_{\R^{n-1}} \int_{0}^{\infty} \int_0^{\infty}  \sum_{\ell=1}^{m_u} \sum_{k=1}^{m_h}
	|h-u|^{p-2}  \\
&\quad\bigg[
    \bar{K}_{\varepsilon,\abs{x'-y'}}( x_{2k}-y_{2\ell}) 
	 -\bar{K}_{\varepsilon,\abs{x'-y'}}( x_{2k}-y_{2\ell-1}) \\
&\quad-\bar{K}_{\varepsilon,\abs{x'-y'}}( x_{2k-1}-y_{2\ell}) 
	 +\bar{K}_{\varepsilon,\abs{x'-y'}}( x_{2k-1}-y_{2\ell-1})
	\bigg]\\
 &\quad\left(\sgn(x_{2k}+ x_{2k-1}) - \sgn(y_{2\ell}+y_{2\ell-1})\right)
	\,  dh \, du \, dx' \, dy' \le 0
\end{aligned}
\end{equation}
where $f(x',x_i(h)) = h$, $f(y',y_{i}(u)) = u$ for $i=1,\dots,2m$ and $m_h = m(x',h)$, $m_u = m(y',u)$.
\end{prop}
Even more, we can explicitly write down the derivative in real variables.

\begin{notation}\label{E-notation}
For an interval $I =(a,b) \subset \R$, we write 
$I = I^{+}$ if $a+b>0$, $I = I^0$ if $a+b = 0$, and $I = I^-$ if $a+b<0$. 
For a fixed $x \in \R^n$, we write the sections as a finite union of intervals
\[
U_{x'}^{h}
    = \bigcup_{i=1}^m I_i
    = \bigg(\bigcup_{i=1}^{m_+} I_i^+\bigg)
    \cup
    \bigg(\bigcup_{i=1}^{m_0} I_i^0\bigg)
    \cup
    \bigg(\bigcup_{i=1}^{m_-} I_i^-\bigg), \quad \hbox{for}~h = f(x)
\]
and where $m= m_+ + m_0 + m_-$. We denote $E_+, E_0,E_- \subset \partial U_{x'}^{f(x)}$ as the set of points that belong to a piece of a boundary of a level set that is moving to the left, centered, and moving to the right, respectively, see Figure \ref{fig:E-sets}. More precisely,
\begin{align*}
E_+ = \{ x : x_n \in \partial I^+ \}, \quad
E_0 = \{ x : x_n \in \partial I^0 \}, \quad
E_- = \{ x : x_n \in \partial I^- \}.
\end{align*}
\end{notation}

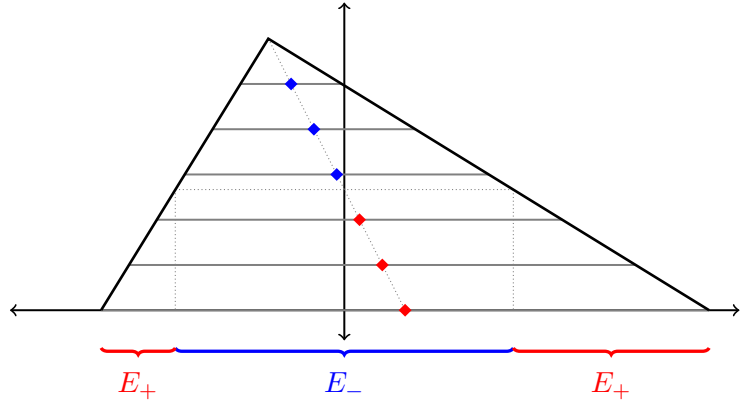
\begin{figure}[htbp]
\begin{tikzpicture}[scale=0.4, use Hobby shortcut, closed=false]
 \tikzset{bullet/.style={diamond,fill,inner sep=1.25pt}}
\draw[<->,line width=.75pt] (-1,-1)--(-1,10.2); 
\draw[<->,line width=.75pt] (-12,0)--(12,0); 
\draw[densely dotted,gray]  (-3.5,9)--(1,0); 
\draw[densely dotted,gray]  (-6.56,4)--(4.56,4); 
\draw[gray, densely dotted] (-6.56,4)-- (-6.56,0); 
\draw[gray, densely dotted] (4.56,4)-- (4.56,0); 
\draw[line width=.75pt, gray] (-9,0)--(11,0); 
	\draw[red] (1,0) node[bullet] {};
\draw[line width=.75pt, gray] (-8.08,1.5)--(8.58,1.5); 
	\draw[red] (.25,1.5) node[bullet] {};
\draw[line width=.75pt, gray] (-7.17,3)--(6.17,3); 
	\draw[red] (-.5,3) node[bullet] {};
\draw[line width=.75pt, gray] (-6.25,4.5)--(3.75,4.5); 
	\draw[blue] (-1.25,4.5) node[bullet] {};
\draw[line width=.75pt, gray] (-5.33,6)--(1.33,6); 
	\draw[blue] (-2,6) node[bullet] {};	
\draw[line width=.75pt, gray] (-4.42,7.5)--(-1.08,7.5); 
	\draw[blue] (-2.75,7.5) node[bullet] {};	
\draw[line width=1pt] (-9,0)--(-3.5,9)--(11,0); 
\draw [ line width=1.25pt,decoration={brace,mirror,raise=0.5cm}, decorate, blue] (-6.56,0)--(4.56,0); 
	\node[blue] at (-1,-2.5) {$E_-$};
\draw [ line width=1.25pt,decoration={brace,mirror,raise=0.5cm}, decorate, red] (-9,0)--(-6.56,0); 
	\node[red] at (-7.79,-2.5) {$E_+$};
\draw [ line width=1.25pt,decoration={brace,mirror,raise=0.5cm}, decorate, red] (4.56,0)--(11,0); 
	\node[red] at (7.78,-2.5) {$E_+$};
\end{tikzpicture}
\caption{Decomposition of $\supp f$ based on the center of mass of the level sets.}
\label{fig:E-sets}
\end{figure}

The following corollary of Proposition \ref{lem:deriv for good} follows from undoing the change of variables. 
\begin{cor} \label{lem:changeback-nonlocal}
Assume that $f$ is a good function, then 
\begin{equation}\label{eq:antisymmetry energy}
\begin{aligned}
\frac{d}{d\tau}\mathcal{F}_{\varepsilon}^p(f^\tau) \bigg|_{\tau=0}=2p (p-1) &\bigg(
   \int_{E_+}\int_{E_- \cup E_0} \bar{K}_{\varepsilon,|x'-y'|}(x_n-y_n)f_{x_n}(x) f_{y_n}(y) \, dx \, dy \\
   &\quad- \int_{E_-}\int_{E_+ \cup E_0} \bar{K}_{\varepsilon,|x'-y'|}(x_n-y_n)f_{x_n}(x) f_{y_n}(y) \, dx \, dy
\bigg).
\end{aligned}
\end{equation}
\end{cor}

One can view the expression on the right-hand side of  \eqref{eq:antisymmetry energy} as a quantification of asymmetry. 
Indeed, if $f$ is radially symmetric across $\{x_n=0\}$, then $E_+ = E_- = \emptyset$ and the derivative is zero. 
Note that if $|E_0| = 0$, then \eqref{eq:antisymmetry energy} can be written concisely as
 \begin{equation}\label{eq:No-E0}
\frac{d}{d\tau}\mathcal{F}_{\varepsilon}^p(f^\tau) \bigg|_{\tau=0}=4p (p-1) 
   \int_{E_+}\int_{E_-} \bar{K}_{\varepsilon,|x'-y'|}(x_n-y_n)f_{x_n}(x) f_{y_n}(y) \, dx.
 \end{equation}
 Unfortunately the integrand of this expression does not have a clear sign, and the sign only appears when we are considering the level set representation. We illustrate this in the following example.

 \begin{ex}\label{ex:zero-part 2}
Consider the good function $f:\R^2 \to \R$ given in Example \ref{ex:zero} and let $B_1(0,t) = \{(x',x_n) \in \R^{n-1} \times \R : |(x', x_n-t)| < 1\}$. 
From \eqref{eq:No-E0}, we have
\begin{align*}
&\frac{d}{d\tau}\mathcal{F}_{\varepsilon}^p(f^\tau) \bigg|_{\tau=0}\\
   &= 4p (p-1) \int_{B_1(0,2)} \int_{B_1(0,-2)} \bar{K}_{\varepsilon,|x'-y'|}(x_n-y_n)
   \frac{x_n+2}{|(x',x_n+2)|}
   \frac{y_n-2}{|(y',y_n-2)|} \, dx \, dy \\
   &= -4p (p-1)\int_{B_1(0,2)} \int_{B_1(0,-2)} 
   \int_0^{y_n-x_n} K_{\varepsilon,|x'-y'|}(r)
   \frac{x_n+2}{|(x',x_n+2)|}
   \frac{y_n-2}{|(y',y_n-2)|}  dr \, dx \, dy
\end{align*}
since $x_n<0<y_n$ for all $|(x', x_n+2)|< 1$, $|(y', y_n-2)|< 1$. 
Notice that the sign of the integrand in this expression is not positive for each fixed $(x', x_n)$, $(y', y_n)$, and $r$ in the domain of integration. 
The sign is instead observed by studying the endpoints of the level sets as done in the proof of Proposition~\ref{lem:deriv for good}. 
\end{ex}

\subsubsection{Proof of Proposition~\ref{lem:deriv for good}}
We begin with an expression for the fractional energy in terms of its level sets.  
\begin{lem}\label{lem:energy for good}
Assume that $f$ is a non-negative, good function. Then, 
\begin{align*}
\mathcal{F}_{\varepsilon}^p(f)- C_\varepsilon \|f\|^p_{L^p(\R^n)}
&= p(p-1) \int_{\R^{n-1}} \int_{\R^{n-1}} \int_{0}^{\infty} \int_0^{\infty}\sum_{\ell=1}^{m_u} \sum_{k=1}^{m_h}|h-u|^{p-2}\\ 
	&\quad  \bigg[
    \bar{\bar{K}}_{\varepsilon, |x'-y'|}(x_{2k}-y_{2\ell})
    - \bar{\bar{K}}_{\varepsilon, |x'-y'|}(x_{2k}-y_{2\ell-1})\\
    &\quad  -\bar{\bar{K}}_{\varepsilon, |x'-y'|}(x_{2k-1}-y_{2\ell})
    +\bar{\bar{K}}_{\varepsilon, |x'-y'|}(x_{2k-1}-y_{2\ell-1})
    \bigg]\, dh \, du \, dx' \, dy'
\end{align*}
where $f(x',x_i(h)) = h$, $f(y',y_{i}(u)) = u$ for $i=1,\dots,2m$ and $m_h = m(x',h)$, $m_u = m(y',u)$.
\end{lem}

\begin{proof}
As in \eqref{eq:p-interaction}, we begin by writing
\begin{align*}
\mathcal{F}^p_{\varepsilon}(f)- C_\varepsilon \|f\|^p_{L^p(\R^n)}
    &= -p(p-1)\int_{\R^{2(n-1)}}  \int_{\R^2_+} |h-u|^{p-2} \\
    &\quad\int_{\R^{2}}\chi_{U_{x'}^h}(x_n) \chi_{U_{y'}^h}(y_n)
        K_{\varepsilon,|x'-y'|}(x_n-y_n)  dx_n \, dy_n \, dh \, du \, dx' \, dy'.
\end{align*}
Fix $x',y' \in \R^{n-1}$ and let $h,u>0$. Since $f$ is a good function, there are at most $m_h$ solutions to $f(x',x_n) = h$ which we denote by $x_{2k-1}\leq x_{2k}$, $k=1,\dots, m_h$. 
We similarly denote the solutions to $f(y', y_n) = u$ by $y_{2\ell-1}\leq y_{2\ell}$, $\ell = 1,\dots m_u$. 
Then we have
\begin{align*}
\int_{\R^{2}}\chi_{U_{x'}^h}(x_n) &\chi_{U_{y'}^h}(y_n)
        K_{\varepsilon,|x'-y'|}(x_n-y_n)  \, dx_n \, dy_n\\
    &= \sum_{\ell=1}^{m_u} \sum_{k=1}^{m_h} \int_{y_{2\ell-1}}^{y_{2\ell}} \int_{x_{2k-1}}^{x_{2k}} K_{\varepsilon, |x'-y'|}(x_n-y_n) \, dx_n \, dy_n\\
    &=  \sum_{\ell=1}^{m_u} \sum_{k=1}^{m_h} \int_{y_{2\ell-1}}^{y_{2\ell}} \big[
    \bar{K}_{\varepsilon, |x'-y'|}(x_{2k}-y_n)
    -\bar{K}_{\varepsilon, |x'-y'|}(x_{2k-1}-y_n)
    \big]\, dy_n\\
    &= - \sum_{\ell=1}^{m_u} \sum_{k=1}^{m_h} \big[
    \bar{\bar{K}}_{\varepsilon, |x'-y'|}(x_{2k}-y_{2\ell})
    - \bar{\bar{K}}_{\varepsilon, |x'-y'|}(x_{2k}-y_{2\ell-1})\\
    &\quad -\bar{\bar{K}}_{\varepsilon, |x'-y'|}(x_{2k-1}-y_{2\ell})
    +\bar{\bar{K}}_{\varepsilon, |x'-y'|}(x_{2k-1}-y_{2\ell-1})
    \big],
\end{align*}
as desired.
\end{proof}

We are now prepared to present the proof of Proposition \ref{lem:deriv for good}. 

\begin{proof}[Proof of Proposition~\ref{lem:deriv for good}]
First fix $\varepsilon>0$. 
Since $f$ is a good function, we know that $f^\tau$ is also a good function. 
By Lemma \ref{lem:energy for good} for $f^\tau$ and applying \eqref{eq:preserves Lp}, we have
\begin{align*}
\mathcal{F}_{\varepsilon}^p(f^\tau)
= C_\varepsilon \|f\|^p_{L^p(\R^n)}
    &+p(p-1) \int_{\R^{n-1}} \int_{\R^{n-1}} \int_{0}^{\infty} \int_0^{\infty}\sum_{\ell=1}^{m_u} \sum_{k=1}^{m_h}|h-u|^{p-2}\\ 
	&\quad  \bigg[
    \bar{\bar{K}}_{\varepsilon, |x'-y'|}(x_{2k}^\tau-y_{2\ell}^\tau)
    - \bar{\bar{K}}_{\varepsilon, |x'-y'|}(x_{2k}^\tau-y_{2\ell-1}^\tau)\\
    &\quad  -\bar{\bar{K}}_{\varepsilon, |x'-y'|}(x_{2k-1}^\tau-y_{2\ell}^\tau)
    +\bar{\bar{K}}_{\varepsilon, |x'-y'|}(x_{2k-1}^\tau-y_{2\ell-1}^\tau)
    \bigg]\, dh \, du \, dx' \, dy'
\end{align*}
From the Definition \ref{defn:cont set}(1), we have
\[
\frac{d}{d\tau}(x_i^\tau - y_j^\tau)
    = - (\sgn(x_{2k} + x_{2k-1}) - \sgn(y_{2\ell}+ y_{2\ell-1})), \quad \hbox{for}~i = 2k,2k-1,~j=2\ell,2\ell-1. 
\]
Therefore, taking the derivative of $\mathcal{F}^p_{\eps}(f^\tau)$ with respect to $\tau$ gives
\begin{align*}
\frac{d}{d\tau}\mathcal{F}_{\varepsilon}^p(f^\tau)
&= p(p-1) \int_{\R^{2(n-1)}} \int_{0}^{\infty} \int_0^{\infty} \sum_{k=1}^m \sum_{\ell=1}^m|h-u|^{p-2}\\ 
	&\quad  \frac{d}{d\tau}\bigg[ 
	  \bar{\bar{K}}_{\varepsilon, |x'-y'|}( x_{2k}^\tau-y_{2\ell}^\tau)
	  -  \bar{\bar{K}}_{\varepsilon, |x'-y'|}( x_{2k}^\tau-y_{2\ell-1}^\tau)\\
	&\quad  -  \bar{\bar{K}}_{\varepsilon, |x'-y'|}( x_{2k-1}^\tau-y_{2\ell}^\tau)
	+   \bar{\bar{K}}_{\varepsilon, |x'-y'|}( x_{2k-1}^\tau-y_{2\ell-1}^\tau)
	\bigg] \, dh \, du \, dx' \, dy'\\
&= -p(p-1)\int_{\R^{2(n-1)}} \int_{0}^{\infty} \int_0^{\infty} \sum_{k=1}^m \sum_{\ell=1}^m|h-u|^{p-2}\\ 
	&\quad  \bigg[ 
	 \bar{K}_{\varepsilon, |x'-y'|}( x_{2k}^\tau-y_{2\ell}^\tau)
	  - \bar{K}_{\varepsilon, |x'-y'|}( x_{2k}^\tau-y_{2\ell-1}^\tau)\\
	&\quad  - \bar{K}_{\varepsilon, |x'-y'|}( x_{2k-1}^\tau-y_{2\ell}^\tau)
	+  \bar{K}_{\varepsilon, |x'-y'|}( x_{2k-1}^\tau-y_{2\ell-1}^\tau)
	\bigg] \\
 &\quad (\sgn(x_{2k}+x_{2k-1}) - \sgn(y_{2\ell}+y_{2\ell-1})) \, dh \, du \, dx' \, dy'.
\end{align*}
Evaluating at $\tau=0$, we obtain \eqref{eq:deriv level sets}. 

We next show that the integrand in \eqref{eq:deriv level sets} is positive for each fixed $x', y', h,u, k,\ell$. For this, we use the simplified notation 
\begin{equation}\label{eq:+-notation}
x_+ = x_{2k}, \quad x_-= x_{2k-1},\quad y_+ = y_{2\ell}, \quad y_-= y_{2\ell-1}
\end{equation}
and $K_{\varepsilon}(r) = K_{\varepsilon,\abs{x'-y'}}(r)$.
Assume, without loss of generality, that 
\[
(x_{+}+x_-)-(y_{+}+y_{-}) > 0.
\]
Then, it is enough to check that
\begin{equation}\label{eq:K-claim}
\big[\bar{K}_{\varepsilon}( x_{+}-y_{+}) 
	 -\bar{K}_{\varepsilon}( x_{+}-y_{-}) -\bar{K}_{\varepsilon}( x_{-}-y_{+}) 
	 +\bar{K}_{\varepsilon}( x_{-}-y_{-}) \big] >0.
\end{equation}
We break into three cases based on the interaction of the intervals $(x_-,x_+)$ and $(y_-,y_+)$. 
In the following, we will use the antisymmetry of $\bar{K}_{\varepsilon}(r)$ and that $K'_{\varepsilon}(r)<0$ for $r>0$. 

\medskip

{\bf Case 1}. Embedded intervals: $(y_-,y_+) \subset (x_-,x_+)$.

\medskip

\noindent
Since $x_- < y_-<y_+<x_+$, we have
\begin{align*}
\bar{K}_{\varepsilon}( x_{+}-y_{+}) -&\bar{K}_{\varepsilon}( x_{+}-y_{-}) 
	-\bar{K}_{\varepsilon}( x_{-}-y_{+}) +\bar{K}_{\varepsilon}( x_{-}-y_{-})\\
&= \int_{-(x_{-}-y_{-})}^{ x_{+}-y_{+}} K_{\varepsilon}(r) \, dr 
	- \int_{-(x_{-}-y_{+})}^{  x_{+}-y_{-}} K_{\varepsilon}(r) \, dr\\
&= \int_{0}^{ (x_{+}+x_-)-(y_{+}+y_{-})} \big[K_{\varepsilon}(r+(y_--x_{-})) - K_{\varepsilon}(r+(y_+-x_{-})) \big] \, dr\\
&= \int_{0}^{ (x_{+}+x_-)-(y_{+}+y_{-})}  \int_{y_+-x_-}^{y_--x_-} K'_{\varepsilon}(r+\xi) \, d \xi \,  dr\\
&=- \int_{0}^{ (x_{+}+x_-)-(y_{+}+y_{-})}  \int_{0}^{y_+- y_-} K'_{\varepsilon}(r+\xi+y_--x_-) \, d\xi \, dr \\
&{\geq \big[(x_++x_-)-(y_++y_-)\big] (y_+-y_-) \min_{y_--x_- < r < x_+-y_-} |K'_{\varepsilon}(r)|}
> 0.
\end{align*}

\medskip

{\bf Case 2}. Separated intervals: $(y_-,y_+) \cap (x_-,x_+) = \emptyset$.

\medskip

\noindent
Since $x_- < y_-<y_+<x_+$, we have
\begin{align*}
\bar{K}_{\varepsilon}( x_{+}-y_{+}) -&\bar{K}_{\varepsilon}( x_{+}-y_{-}) 
	-\bar{K}_{\varepsilon}( x_{-}-y_{+}) +\bar{K}_{\varepsilon}( x_{-}-y_{-})\\
&= \int_{x_{-}-y_{+}}^{x_{+}-y_{+}} K_{\varepsilon}(r) \, dr - \int_{x_--y_-}^{x_{+}-y_{-}} K_{\varepsilon}(r) \, dr\\
&= \int_{0}^{x_+-x_-} \big[ K_{\varepsilon}(r+x_{-}-y_{+})- K(r+x_--y_-) \big]\, dr\\
&= \int_{0}^{x_+-x_-} \int_{x_--y_-}^{x_{-}-y_{+}} K_{\varepsilon}'(r+\xi) \, d \xi \, dr\\
&= -\int_{0}^{x_+-x_-} \int^{y_+-y_-}_{0} K_{\varepsilon}'(r+\xi+x_{-}-y_{+}) \, d \xi \, dr \\
&{\geq (x_+-x_-) (y_+-y_-) \min_{x_--y_+ < r < x_+-y_-} |K_{\varepsilon}'(r)|}> 0.
\end{align*}

\medskip

{\bf Case 3}. Overlapping intervals: $(y_-,y_+) \not\subset (x_-,x_+)$ and $(y_-,y_+) \cap (x_-,x_+) \not=\emptyset$. 

\medskip

\noindent
Since $y_-<x_-<y_+<x_+$, we have
\[
x_+-y_- = (x_+-y_+) + (y_+-x_-)+(x_--y_-),
\]
so that
\begin{align*}
\bar{K}_{\varepsilon}&( x_{+}-y_{+}) -\bar{K}_{\varepsilon}( x_{+}-y_{-}) 
	-\bar{K}_{\varepsilon}( x_{-}-y_{+}) +\bar{K}_{\varepsilon}( x_{-}-y_{-})\\
&= \int_0^{x_+-y_+} K_{\varepsilon}(r) \, dr 
	+\int_0^{y_+-x_-} K_{\varepsilon}(r) \, dr 
	+  \int_0^{x_--y_-} K_{\varepsilon}(r) \, dr
	 -\int_0^{x_+-y_-} K_{\varepsilon}(r) \, dr \\
&= \int_0^{y_+-x_-} K_{\varepsilon}(r) \, dr 
	+  \int_0^{x_--y_-} K_{\varepsilon}(r) \, dr
	-\int_{x_+-y_+}^{x_+-y_-} K_{\varepsilon}(r) \, dr \\
&= \int_0^{y_+-x_-} K(r) \, dr 
	+  \int_0^{x_--y_-} K_{\varepsilon}(r) \, dr
	-\int_{0}^{y_+-y_-} K_{\varepsilon}(r+x_+-y_+) \, dr \\
&= \int_0^{y_+-x_-} K_{\varepsilon}(r) \, dr 
	+  \int_0^{x_--y_-} K_{\varepsilon}(r) \, dr\\
    &\quad-\int_{0}^{x_--y_-} K_{\varepsilon}(r+x_+-y_+) \, dr- \int_{x_--y_-}^{y_+-y_-} K_{\varepsilon}(r+x_+-y_+) \, dr \\
&= \int_0^{y_+-x_-} K_{\varepsilon}(r) \, dr 
	+  \int_0^{x_--y_-} K_{\varepsilon}(r) \, dr\\
&\quad
	-\int_{0}^{x_--y_-} K_{\varepsilon}(r+x_+-y_+) \, dr- \int_{0}^{y_+-x_-} K_{\varepsilon}(r+x_+-y_++x_--y_-) \, dr \\
&= -\bigg[\int_0^{y_+-x_-} \int_0^{(x_++x_-)-(y_++y_-)} K_{\varepsilon}'(r+\xi) \, d\xi \, dr
	+ \int_0^{x_--y_-} \int_0^{x_+-y_+} K_{\varepsilon}'(r+\xi) \, d\xi \, dr \bigg]\\
&> \big[(x_++x_-)-(y_++y_-)\big] (y_+-x_-) \min_{0 < r < x_+-y_-} |K_{\varepsilon}'(r)|\\
&\quad + (x_--y_-) (x_+-y_+) \min_{0 < r < (x_++x_-)-(y_+-y_-)} |K_{\varepsilon}'(r)|
=0.
\end{align*}
We have established \eqref{eq:K-claim} in all possible cases. 
Recalling \eqref{eq:K-eps-prime} and a further analysis of the final expressions in Cases 1-3 shows that they are monotone as $\eps\to 0^+$. Moreover,  $\mathcal{F}_\eps^p(f)\to \mathcal{F}^p(f)$ as $\eps \to 0^+$.  Hence, in this case, we can conclude that 
$$
\lim_{\eps \to 0^+} \frac{d}{d\tau}\mathcal{F}_\eps^p(f^\tau)\bigg|_{\tau=0}= \frac{d}{d\tau}\mathcal{F}^p(f^\tau)\bigg|_{\tau=0},
$$
and that the sign of the derivative is preserved.
\end{proof}

\section{Interpolation between symmetric decreasing functions}\label{sec:interp}
Let us now review the interpolation between symmetric decreasing functions of unit mass introduced in \cite{delgadino2022uniqueness}. 
Consider a nonnegative, symmetric decreasing function $f \in L^1(\R^n)\cap C(\R^n)$ with mass 1. 
The associated height function $H:(0,1) \to (0,\|f\|_{L^{\infty}(\R^n)})$ is defined implicitly by
\begin{equation}\label{eq:defH}
    \int_{\R^n} \min \{ f(x), H(m)\} \, dx = m, \quad \hbox{for}~m \in (0,1).
\end{equation}
That is, $H(m)$ is the unique value such that the mass under the plane $f(x) = H(m)$ has mass $m \in (0,1)$, see Figure \ref{fig:height}.
The height function $H$ satisfies the following properties. 

\begin{figure}[ht]
\begin{tikzpicture}[scale=3]
  \fill[gray!20] (-1,0) -- plot[domain=-1:-.57] (\x, {(1 - abs(\x^2))^(1.25)}) |- (-.57,0) -- cycle;
    \fill[gray!20] (1,0) -- plot[domain=1:.57] (\x, {(1 - abs(\x^2))^(1.25)}) |- (.57,0) -- cycle;
      \fill[gray!20] (-.58,0) -- (-.58,.6)--(.58,.6)-- (.58,0) -- cycle;   
  \draw[blue, domain=-1:1, samples=100, line width=1pt] plot (\x, {(1 - abs(\x^2))^(1.25)});
  \draw[cyan, line width=1pt, dashed] (-.57,0.6) -- (.57,0.6);
  	\draw (1,.62) node {\small $f(x) = H(m)$};
    \draw[cyan, line width=1pt, dashed] (.57,0) -- (.57,0.6);
  	\draw (.57,-.1) node {\small $x$};
  \draw[line width=.75pt, <->] (-1.25,0) -- (1.25,0);
  \draw[line width=.75pt, <->] (0,-0.15) -- (0,1.15);
  \draw (0,.3) node {\small mass $=m$};  
\end{tikzpicture}
\caption{The height function $H(m)$ associated to $f(x)$}
\label{fig:height}
\end{figure}
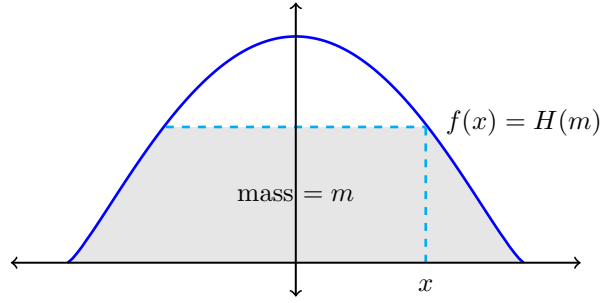

\begin{lem}[See \cite{delgadino2022uniqueness}] \label{lem:properties}
Let $f\in L^1(\R^n)\cap C(\R^n)$ and $H$ its associated height function defined in \eqref{eq:defH}, then
\begin{enumerate}
    \item $H = H(m) \in (0,\|f\|_{L^{\infty}(\R^n)})$ is continuous, strictly increasing, and convex on $(0,1)$. 
    \item If in addition we assume that $f$ has compact support and is strictly decreasing in the radial variable, then
    $$
    \lim_{m\to 0^+}H'(m)=|\{f>0\}|^{-1}\qquad\mbox{and}\qquad\lim_{m\to 1^-}H'(m)=+\infty.
    $$
\item The function $H$ fully determines $f$ as
\begin{equation}\label{h_rho}
f(x)=\int_0^1\chi_{(c_n H'(m))^{-1/n}}(x) H'(m)\;dm, \quad \hbox{where}~\chi_r(x) := \chi_{B(0,r)}(x).
\end{equation}

\item For almost every $m \in (0,1)$, it holds that
\begin{equation}\label{h''}
 \displaystyle- \Phi'\big((c_n H'(m))^{-1/n}\big)=nc_n^{1/n}\frac{\big(H'(m)\big)^{2+\frac{1}{n}}}{H''(m)}\ge 0
\end{equation}
where $\Phi: [0,\infty) \to [0,\infty)$ satisfies $f(x) = \Phi(|x|)$. 
\end{enumerate}

\end{lem}

\begin{proof}
Properties (1)-(3) are established in \cite{delgadino2022uniqueness}*{Lemma 2.1}. Property (4) is stablished in \cite{delgadino2022uniqueness}*{Lemma 4.2} follows from differentiating the identity
$$
\Phi((c_n H'(m))^{-1/n})=H(m).
$$
\end{proof}

Now, consider two symmetric decreasing functions $f_0, f_1 \in L^1(\R^n)\cap C(\R^n)$, both with unit mass and let $H_0, H_1$ denote their associated height functions. 
Let $H_t$, $0 \leq t \leq 1$, be a linear interpolation between $H_0$ and $H_1$:
\begin{equation}\label{eq:interpolation}
H_t(m) = (1-t)H_0(m) + tH_1(m).
\end{equation}
By Lemma \ref{lem:properties}(3), the height function $H_t$  uniquely determines a radially decreasing function $f_t \in L^1(\R^n)\cap C(\R^n)$ with unit mass. 
In particular, $\{f_t\}_{t\in[0,1]}$ is an interpolation between $f_0$ and $f_1$. 

It is shown in \cite{delgadino2022uniqueness}*{Proposition 2.3} that the $p$-th power of the $L^p$ norms are convex under this interpolation if and only if $p \geq 2$. More precisely, 
\begin{equation}\label{eq:Lp-height}
\frac{d^2}{dt^2} \|f_t\|_{L^p(\R^n)}^p
    \,\begin{cases}
    <0 & \hbox{if}~1 \leq p < 2\\
    =0 & \hbox{if}~p=2\\
    >0 & \hbox{if}~p>2
    \end{cases} \quad\quad \hbox{for}~0 < t < 1. 
\end{equation}
In our next result, we determine when the $W^{1,p}$ seminorms are convex under the height function interpolation. 

\begin{prop}\label{prop:interW1p}
Fix $1 < p < \infty$. 
Let $f_0,\,f_1 \in W^{1,p}(\R^n)$ be two distinct, non-negative, symmetric decreasing functions with unit mass and let $f_t$, $0 \leq t \leq 1$, be the height function interpolation between $f_0$ and $f_1$. 
It holds that
\[
t \mapsto [f_t]_{W^{1,p}(\R^n)}^p = \int_{\R^n} |\nabla f_t|^p \, dx
\]
is convex if $p \geq 2n/(n+1)$.

Consequently, the following a priori estimate on the interpolation holds when $p \geq 2n/(n+1)$
\begin{equation}\label{eq:apriorinterpolation}
    \max_{t\in[0,1]}[f_t]_{W^{1,p}(\R^n)}^p\le \max\left\{[f_0]_{W^{1,p}(\R^n)}^p,[f_1]_{W^{1,p}(\R^n)}^p\right\}.
\end{equation}
\end{prop}

\begin{rem}
It is not known if the condition $p \geq 2n/(n+1)$ is sharp for the convexity in Proposition \ref{prop:interW1p}.
\end{rem}

\begin{proof}
Given a symmetric, radially decreasing function $f \in W^{1,p}(\R^n)$, 
we let $\Phi: [0,\infty) \to [0,\infty)$ be such that $f(x) = \Phi(|x|)$ and write
\begin{equation}\label{eq:fpPhi}
\int_{\R^n}|\nabla f|^p\;dx
    =n c_{n}\int_0^\infty r^{n-1}|\Phi'(r)|^p\;dr.
\end{equation}
As a consequence of \eqref{h_rho}, the radial variable $r$ can be expressed in terms of the height function $H$ as
$$
r=(c_n H'(m))^{-1/n}
$$
which gives
$$
dr=-\frac{c_n^{-1/n}}{n} \frac{H''(m)}{H'(m)^{1+1/n}} \;dm.
$$
Moreover, using \eqref{h''}, we can write
$$
\Phi'(r)=-nc_n^{1/n}\frac{(H'(m))^{2+1/n}}{H''(m)}.
$$
Therefore, applying the change of variables $r \mapsto m$ in \eqref{eq:fpPhi} gives
\begin{align*}
\int_{\R^n}|\nabla f|^p \, dx
&=n c_n\int_0^1(c_n H'(m))^{-1+1/n} n^pc_n^{p/n}\frac{(H'(m))^{p(2+1/n)}}{(H''(m))^p}\frac{c_n^{-1/n}}{n} \frac{H''(m)}{H'(m)^{1+1/n}} \;dm\\
&=C_{n,p}\int_0^1 (H'(m))^{p(2+1/n)-2} (H''(m))^{1-p}   \;dm.
\end{align*}

We consider the function $\Psi:\R\times [0,\infty) \to [0,\infty)$ given by
$$
\Psi(a,b)=a^{k} b^{\ell} \quad \hbox{for fixed}~k,\ell \in \R.
$$
To check the convexity or concavity of $\Psi$, we find the Hessian matrix
$$
D^2\Psi(a,b)=\begin{pmatrix}
k(k-1) a^{k-2}b^{\ell} & k\ell a^{k-1}b^{\ell-1}\\
k\ell a^{k-1}b^{\ell-1}& \ell(\ell-1)a^{k}b^{\ell-2}
\end{pmatrix}.
$$
Momentarily, we will set 
\begin{equation}\label{eq:kl}
k=p(2+1/n)-2\qquad\mbox{and}\qquad \ell =1-p,
\end{equation}
so we assume that $\ell < 0$ and $k>0$ since $p > 1$. 
Now, we have two cases $0<k<1$ and $1< k$, which corresponds to $1< p<3n/(2n+1)$ and $p> 3n/(2n+1)$, respectively.  
In the case $0 < k < 1$, the first minor is negative, hence if the determinant $\det D^2\Psi$ is positive, then the matrix $D^2\Psi$ is negative definite. 
In the case $1< k$, the first minor is positive, so if the determinant is positive, then the matrix $D^2\Psi$ is positive definite. 

Hence, to determine the convexity or concavity of $\Psi$, we need to check the positivity of the determinant of the Hessian, which is given by
$$
\det(D^2\Psi(a,b))=a^{2(k-1)}b^{2(\ell-1)}(k(k-1)\ell(\ell-1)-k^2\ell^2)=a^{2k-2}b^{2\ell-2}k\ell(1-k-\ell).
$$
Now taking $k,\ell$ as in \eqref{eq:kl}, we find that
\begin{align*}
k\ell(1-k-\ell)
    &=(p(2+1/n)-2)(1-p)(1- (p(2+1/n)-2) - (1-p))\\
    &= (p(2+1/n)-2)(p-1)(p(1+1/n)-2)
\end{align*}
is nonnegative if and only if
$$
p\geq \frac{2n}{n+1}.
$$
The stated result follows after writing
\[
\int_{\R^n}|\nabla f_t|^2 \, dx
    =C_{n,p}\int_0^1 \Psi(H_t'(m), H_t''(m))   \;dm.
\]
\end{proof}

Strict convexity under the height function interpolation also holds for potential energies with symmetric increasing potentials.

\begin{prop}\label{prop:V}
Consider $V$ a smooth, bounded, increasing radially symmetric potential. 
Let $f_0,f_1 \in C(\R^n)$ be two distinct, non-negative, symmetric decreasing functions with unit mass and let $f_t$, $0 \leq t \leq 1$, be the height function interpolation. 
Then 
\[
t \mapsto \int_{\R^n}  V(x) f_t(x)\, dx
\]
is strictly convex for all $0 < t < 1$.
\end{prop}

\begin{proof}
Let $v:[0,\infty) \to \R$ be such that $V(x) = v(|x|)$. 
With this and \eqref{h_rho} for $f_t$, we re-write 
\begin{equation}\label{Potential Energy}
\begin{aligned}
\int_{\R^n}V(x)f_t(x)\;dx
    &=  \int_{\R^n}V(x) \int_0^1\chi_{(c_nH_t'(m))^{-1/n}}(x)H_t'(m)\;dm \, dx\\
    &=  \int_0^1 \left(n 
c_n\int_{0}^{(c_nH_t'(m))^{-1/n}}v(r)r^{n-1}\;dr\right) H_t'(m)\;dm\\
&=\int_0^1 F_V((c_nH_t'(m))^{-1/n})H_t'(m)\;dm,
\end{aligned}
\end{equation}
where we define
$$
F_V(\xi):=nc_n\int_0^\xi v(r) r^{n-1}\;dr, \quad \xi \geq 0. 
$$
Differentiating this function we obtain
$$
F_V'(\xi)=nc_nv(\xi)\xi^{n-1},
$$
so that differentiating the potential energy \eqref{Potential Energy}, we obtain
\begin{equation}\label{dV}
\begin{aligned}
\frac{d}{dt}&\int_{\R^n}V(x)f_t(x)\;dx\\
&=\int_0^1 \left( -\frac{1}{n}\frac{F'_V((c_nH_t'(m))^{-1/n})}{(c_n H_t'(m))^{1/n}}+F_V((c_n H_t'(m))^{-1/n})\right)(H_1'(m)-H_0'(m))\;dm\\
&=\int_0^1\left(- \frac{v((c_nH_t'(m))^{-1/n})}{H_t'(m)}+F_V((c_nH_t'(m))^{-1/n})\right)(H_1'(m)-H_0'(m))\;dm.
\end{aligned}
\end{equation}
Differentiating again gives the desired result
\[
 \frac{d^2}{dt^2}\int_{\R^n}V(x)f_t(x)\;dx
 =
\int_0^1 \frac{1}{c_n^{1/n}n}\frac{v'((c_nH_t'(m))^{-1/n})}{(H_t'(m))^{2+1/n}}(H_1'(m)-H_0'(m))^2\;dm >0.
\]
\end{proof}

Lastly, we turn our attention to proof of Theorem \ref{thm:nonlocal-convexity}.
For reference, we state a simplified version of the result in \cite{delgadino2022uniqueness}. 

\begin{prop}[Proposition 4.5 in \cite{delgadino2022uniqueness}]\label{prop:DYY-kernel}
Consider $W$ a smooth, bounded, increasing radially symmetric kernel. 
Let $f_0,f_1 \in C(\R^n)$ be two distinct, non-negative, symmetric decreasing functions with unit mass and let $f_t$, $0 \leq t \leq 1$, be the height function interpolation. 
Then 
\[
t \mapsto \int_{\R^n} \int_{\R^n} f_t(x) f_t(y) W(x-y) \, dx \, dy
\]
is strictly convex for all $0 < t < 1$. Moreover, the convexity is monotonic on the derivative with respect to the radial variable $W'$ of the potential.
\end{prop}

While the original result in \cite{delgadino2022uniqueness} allows for more singular kernels, the kernels in the Gagliardo seminorms are not included. 
We again utilize the $\varepsilon$-regularization in \eqref{eq:Fep}. 

\begin{proof}[Proof of Theorem \ref{thm:nonlocal-convexity}]
Fix $\eps>0$ and let $\mathcal{F}_\eps^2$ be as in \eqref{eq:Fep}. 
Recalling \eqref{eq:Iep}, we write
\[
\mathcal{F}_{\varepsilon}^2(f_t)
    =C_{\varepsilon}\|f_t\|_{L^2(\R^n)}^2\\
    -2 \int_{\R^{2n}} f_t(x) f_t(x) W_{\varepsilon}(x-y) \, dx \, dy.
\]
Note that $\tilde{W}_\eps = -2W_\eps$ is a smooth, bounded, increasing radially symmetric kernel. 
Consequently, we may apply \eqref{eq:Lp-height} and Proposition \ref{prop:DYY-kernel} to obtain 
\[
\frac{d^2}{dt^2}\mathcal{F}_{\varepsilon}^2(f_t)
    = 0 + \frac{d^2}{dt^2} \bigg[\int_{\R^{2n}} f_t(x) f_t(x) \tilde{W}_\eps(x-y) \, dx \, dy\bigg]>0.  
\]
The convexity now follows by taking $\varepsilon \to 0^+$, using the monotonicity with respect of $\eps$ of $W'_\eps$, and the monotonicity of convexity under Proposition \ref{prop:DYY-kernel}. 
\end{proof}

\section{On Theorem \ref{thm:uniqueness}}\label{sec:uniqueness}

In this section, we use Theorem \ref{thm:main2} and Theorem \ref{thm:nonlocal-convexity} to establish Theorem \ref{thm:uniqueness}. 
Notice from Figures \ref{fig:symmetrization} and \ref{fig:local-ex} that the continuous Steiner symmetrization $v^\tau$ of $v$ does not preserve the support of $v$, so we cannot directly compare for $v$ and $v^\tau$ using \eqref{eq:stationary} to establish uniqueness. 
To preserve the support of $v$, we slow down the speed of the level sets near $h=0$ in Definition \ref{defn:cont set}. 

Before proceeding with the proof of Theorem \ref{thm:uniqueness}, 
we present truncated continuous Steiner symmetrizations and their properties. 

\subsection{Truncated symmetrizations}\label{sec:truncated}

Fix $h_0>0$ and let $v_0(h) = \min\{1, \frac{h}{h_0}\}$ for $h \geq 0$. 
The continuous Steiner symmetrization truncated at height $h_0$ of a super-level set $U = \{f>h\} \subset \R$ of a function $f$ at height $h>0$ is given by $M^{v_0(h)\tau}(U)$. 
The continuous Steiner symmetrization truncated at height $h_0$ of a nonnegative function $f \in L^1(\R^n)$ in the direction of $e_n$ is denoted by $\tilde{f}^\tau$ and defined as
\[
\tilde{f}^\tau(x) = \int_0^{\infty} \chi_{M^{v_0(h)\tau}(U_{x'}^h)}(x_n) \, dh \quad \hbox{for}~x = (x',x_n) \in \R^{n},~h>0.
\]

\begin{figure}[htbp]
 \begin{tikzpicture}[scale=1.75, use Hobby shortcut, closed=false]
 \tikzset{bullet/.style={diamond,fill,inner sep=1.25pt}}
\draw[line width=1pt, cyan] 
    (-3.25,0)--(-2.75,0)--(-1.75,2)--(-0.75,0) 
	--(.75,0)--(1.75,1)--(2.75,0)--(3.25,0); 
\draw[cyan] (-1.6,2.15) node { $f^\tau$};
 \draw[line width=1pt, blue] 
    (-3.25,0)--(-3,0)--(-2.63,.25)--(-1.75,2)--(-1,.5);
\draw[line width=1pt, blue] (1,.25)--(1.75,1)--(2.5,.25)--(3,0)--(3.25,0); 
\draw[blue] (-1.3,2.15) node { $\tilde{f}^\tau$};
  \draw[line width=1pt, blue] 
    (-1,.5)-- (-.87,0);
\draw[line width=1pt] 
    (-3.25,0)--(-3,0)--(-2,2)--(-1,0) 
	--(1,0)--(2,1)--(3,0)--(3.25,0); 
\draw (-1.9,2.15) node { $f$};
 \draw[line width=1.25pt, red, ->] (-.75,.4)--(-.75,.1);
 \draw[line width=1pt,red] (1,0)--(1,.25); 
\draw[line width=1pt,<->] (0,-.25)--(0,2.25); 
	\draw (.25,2.25) node {$h$};
\draw[line width=1pt,<->] (-3.5,0)--(3.5,0); 
	\draw (3.75,0) node {$y$};
 \draw[line width=.75pt,red] (-3.25,.25)--(3.25,.25); 
	\draw[red] (3.75,.25) node {$h=h_0$};
\end{tikzpicture}
\caption{The graphs of $f$, $f^\tau$, and $\tilde{f}^\tau$ in Example \ref{ex:zero} at $x=0$ with $h_0 =\tau= .25$, illustrating how the level sets below the line $h=h_0$ have dropped}
\label{fig:truncated}
\end{figure}
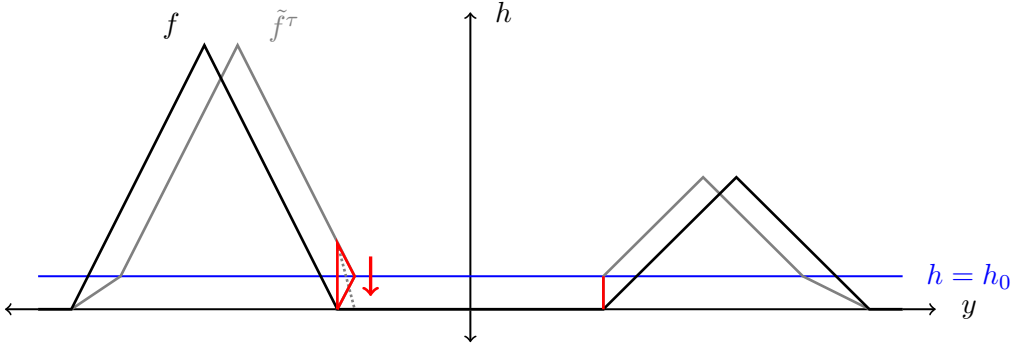

Given a Lipschitz function $f$, we know by Corollary \ref{cor:Lip} that $f^\tau$ is also Lipschitz. 
However, the corresponding truncated symmetrization, $\tilde{f}^\tau$, is not necessarily Lipschitz for all $\tau$ since the level sets near $h=0$ move slower than those above.
In particular, the higher level sets may `drop', {see Figure \ref{fig:truncated}}. 
We will show that, when $\tau$ is sufficiently small, this does not happen and that $\tilde{f}^\tau$ is Lipschitz with the same support as $f$. 

\begin{prop}\label{prop:Lip constant}
Let $f:\R^n \to [0,\infty)$ be Lipschitz with $c_0 = [f]_{\text{Lip}}$.
Then, for each $0 \leq \tau < h_0/c_0$, the function $\tilde{f}^\tau$  is Lipschitz with 
\begin{equation}\label{eq:lip truncated}
[\tilde{f}^\tau]_{\text{Lip}} \leq  \frac{c_0h_0}{h_0-c_0 \tau}
\end{equation}
and satisfies
\begin{equation}\label{eq:support}
\supp \tilde{f}^\tau = \supp f\quad \hbox{and} \quad
\tilde{f}^\tau = f^{\tau} \quad \hbox{in}~\{f^{\tau}>h_0\}.
\end{equation}
Consequently, the upper Dini derivative of $\tilde{f}^\tau$ with respect of $\tau$ satisfies
\begin{equation}\label{eq:dini-bound}
\frac{d^+}{d\tau}[\tilde{f}^\tau]_{\text{Lip}} \bigg|_{\tau=0}  \leq \frac{c_0^2}{h_0}. 
\end{equation}
\end{prop}

First, we prove a characterization of Lipschitz functions with respect to their level sets. 

\begin{lem}\label{lem:distance from boundary}
For a function $f:\R^n \to [0,\infty)$, it holds that
\begin{equation}\label{eq:Lip bound}
[f]_{\text{Lip}} = \sup_{0 < h_1 <h_2} \frac{h_2-h_1}{\operatorname{dist}(\partial \{f\geq h_2\},  \partial\{f \leq h_1\})}. 
\end{equation}
\end{lem}

\begin{proof} 
For ease, set
\[
c_0 :=  \sup_{0 < h_1 <h_2} \frac{h_2-h_1}{\operatorname{dist}(\partial \{f\geq h_2\},  \partial\{f \leq h_1\})}. 
\]
We will show that $c_0 = [f]_{\text{Lip}}$. 

First, we claim that $c_0 \leq [f]_{\text{Lip}}$. If $[f]_{\text{Lip}}=+\infty$, there is nothing to show, so assume that $f$ is Lipschitz. 
Fix $0 < h_1 < h_2$. If $x,y \in \R^n$ are such that $f(y) = h_1$ and $f(x) = h_2$, then
\[
h_2 - h_1 = \abs{f(x) - f(y)} \leq [f]_{\text{Lip}}  \abs{x-y}.
\]
Taking the infimum over all $x \in \{ f = h_1\}$ and $y \in \{ f = h_2\}$, we have that
\[
h_2 - h_1 = \abs{f(x) - f(y)} \leq [f]_{\text{Lip}} \operatorname{dist}(\partial \{f\geq h_2\},  \partial\{f \leq h_1\}). 
\]
Equivalently, 
\[
\frac{h_2 - h_1}{\operatorname{dist}(\partial \{f\geq h_2\},  \partial\{f \leq h_1\})}  \leq [f]_{\text{Lip}} \quad \hbox{for all}~0 < h_1 <h_2,
\]
and we have that $c_0 \leq [f]_{\text{Lip}}$. 

Let us now show that $[f]_{\text{Lip}} \leq c_0$. We may assume that $c_0 < \infty$; otherwise we are done.  
Let $x,y \in \R^n$ and set $h_1 = f(y)$, $h_2 = f(x)$. 
Without loss of generality, assume $0 < h_1 < h_2$. Then
\[
\abs{f(x) - f(y)} = h_2-h_1 \leq c_0 \operatorname{dist}(\partial \{f\geq h_2\}, \partial \{f \leq h_1\})  \leq c_0\abs{x-y},
\]
which shows that $f$ is Lipschitz with $[f]_{\text{Lip}} \leq c_0$. This completes the proof of \eqref{eq:Lip bound}.
\end{proof}

\begin{rem}
Following the proof of Lemma \ref{lem:distance from boundary}, one can show for $\alpha \in (0,1)$ that
\[
[f]_{C^\alpha} = \sup_{0 < h_1 <h_2} \frac{h_2-h_1}{|\operatorname{dist}(\partial \{f\geq h_2\},  \partial\{f \leq h_1\})|^{\alpha}}. 
\]
We expect a similar result to Proposition \ref{prop:Lip constant} to hold for all $C^{\alpha}$-seminorms. 
\end{rem}

Note in the following that if $f$ is a good function, then
\[
\partial\{f\geq h\} = \partial \{f>h\} = \partial \{ f \leq h\} \quad \hbox{for}~h>0. 
\]

\begin{proof}[Proof of Proposition \ref{prop:Lip constant}]
By Corollary \ref{cor:Lip}, we have that $f^\tau$ is Lipschitz with $[f^\tau]_{\text{Lip}} \leq c_0$. 
Fix $h_2 > h_1 > 0$. 
By Lemma \ref{lem:distance from boundary}, we have that
\[
\dist( \partial\{f^{\tau}\geq h_2\}, \partial\{f^{\tau}\leq h_1\} ) \geq \frac{h_2-h_1}{c_0} \quad \hbox{for all}~\tau\geq0.
\]
Assume for now that $f$, and consequently $f^\tau$, is a good function, so that
\[
\dist( \partial\{f^{\tau}> h_2\}, \partial\{f^{\tau}> h_1\} ) \geq \frac{h_2-h_1}{c_0} \quad \hbox{for all}~\tau\geq0.
\]
With this and Lemma \ref{lem:distance for fixed h}, we have
\begin{align*}
\operatorname{dist}(& \partial\{f^{\tau v_0(h_2)} \geq h_2\}, \partial\{f^{\tau v_0(h_1)}\leq h_1\}) \\
&= \operatorname{dist}( \partial\{f^{\tau v_0(h_2)} > h_2\}, \partial\{f^{\tau v_0(h_1)}> h_1\})\\
&\geq \operatorname{dist}( \partial\{f^{\tau v_0(h_1)} > h_2\}, \partial\{f^{\tau v_0(h_1)}> h_1\})
     - \operatorname{dist}( \partial\{f^{\tau v_0(h_2)} > h_2\},
     \partial\{f^{\tau v_0(h_1)} > h_2\}) \\
&\geq \frac{h_2-h_1}{c_0} - \abs{v_0(h_2)\tau-v_0(h_1)\tau} \\
&\geq \frac{h_2-h_1}{c_0} - (h_2-h_1)\frac{\tau}{h_0}
=  \left(\frac{c_0h_0}{h_0-c_0\tau}\right)^{-1}(h_2-h_1).
\end{align*}
For each fixed $x' \in \R^{n-1}$, we can similarly show that
\begin{align*}
\dist(\partial \{f^{\tau v_0(h_2)}(x', \cdot) \geq h_2\} , \partial \{f^{\tau v_0(h_1)}(x', \cdot)\leq h_1\} )
    \geq \left(\frac{c_0h_0}{h_0-c_0\tau}\right)^{-1}(h_2-h_1) >0
\end{align*}
for all $\tau <h_0/c_0$. 
Consequently,
\[
M^{v_0(h)\tau}
(U^{h_2}_{x'}) \subset M^{v_0(h)\tau}(U^{h_1}_{x'}) \quad \hbox{for all}~ 0 <h_1 < h_2~\hbox{and}~x' \in \R^{n-1}.
\]
That is, the sections $U_{x'}^h$ remain ordered and we have \eqref{eq:support}. 
Therefore, $\tilde{f}^\tau = f^{\tau v_0(h_1)}$ 
for all $\tau <h_0/c_0$ and $h>0$, so we have
\begin{align*}
\operatorname{dist}( \partial\{\tilde{f}^\tau \geq h_2\}, \partial\{\tilde{f}^\tau \leq h_1\})
    &= \operatorname{dist}( \partial\{f^{\tau v_0(h_2)} \geq h_2\}, \partial\{f^{\tau v_0(h_1)} \leq h_1\})\\
    &\geq \left(\frac{c_0h_0}{h_0-c_0\tau}\right)^{-1}(h_2-h_1).
\end{align*}
It follows from Lemma \ref{lem:distance from boundary} that $\tilde{f}^\tau$ is Lipschitz for $\tau < h_0/c_0$ with \eqref{eq:lip truncated}. 

Suppose now that $f$ is a Lipschitz function but not a good function. 
In light of Lemma \ref{lem:density}, there is an approximating sequence of functions $f_k$ that are both good and Lipschitz with $[f_k]_{\text{Lip}(\R^n)} \leq c_0$. 
By the above, we have that $\tilde{f}_k^\tau$ are also good and Lipschitz.
Consequently,
\begin{align*}
[\tilde{f}^\tau]_{\text{Lip}(\R^n)}
    &\leq [\tilde{f}_k^\tau-\tilde{f}^\tau]_{\text{Lip}(\R^n)}
        + [\tilde{f}_k^\tau]_{\text{Lip}(\R^n)}\\
   & \leq [\tilde{f}_k^\tau-\tilde{f}^\tau]_{\text{Lip}(\R^n)}
        + \frac{c_0h_0}{h_0 - c_0 \tau}
    \to \frac{c_0h_0}{h_0 - c_0 \tau} \quad \hbox{as}~k \to \infty
\end{align*}
and the result holds. 

To prove \eqref{eq:dini-bound}, we simply use \eqref{eq:lip truncated} to estimate
\[
\frac{d^+}{d\tau} [\tilde{f}^\tau]_{\text{Lip}} \bigg|_{\tau=0}
= \limsup_{\tau \to 0^+} \frac{[\tilde{f}^\tau]_{\text{Lip}}- [f]_{\text{Lip}}}{\tau}
\leq \limsup_{\tau \to 0^+} \frac{\frac{c_0h_0}{h_0-c_0\tau}- c_0}{\tau}
=\limsup_{\tau \to 0^+} \frac{c_0^2}{h_0-c_0\tau} = \frac{c_0^2}{h_0}. 
\]
\end{proof}

We will also need the following estimate on the distance between $f$ and $\tilde{f}^\tau$ in $L^1$.  
See \cite{Brock}*{Theorem 4.2} for a similar result in the setting of Remark \ref{rem:Brockspeed}. 

\begin{lem}\label{lem:Lp bound}
Let $f:\R^n \to [0,\infty)$ have compact support.  
If $f \in L^{\infty}(\R^n)$ is Lipschitz with $c_0 = [f]_{\text{Lip}}$, then 
\begin{equation}\label{eq:bound}
    \|f - \tilde{f}^\tau\|_{L^{\infty}(\R^n)} \leq \tau  [f]_{\text{Lip}(\R^n)}  \quad \hbox{for all}~\tau < h_0/c_0.
\end{equation}
Consequently, 
\[
\|f - \tilde{f}^\tau\|_{L^1(\R^n)} \leq \tau  [f]_{\text{Lip}(\R^n)} |\supp f| \quad \hbox{for all}~\tau < h_0/c_0.
\]
Moreover, the same bounds also hold for the standard symmetrization $f^\tau$.
\end{lem}

\begin{proof}
Assume, up to an approximation, that $f$ is a good function. 
Fix $x = (x',x_n) \in \supp f$ and $0 \leq \tau < h_0/c_0$. 
Let $h_1 = \tilde{f}^\tau(x)$ and $h_2 = f(x)$, and assume, without loss of generality, that $0 < h_1 < h_2$. 
Note that there is a $y_n \in \R$ such that $\tilde{f}^\tau(x', x_n) = f(x', y_n)$, which implies that $\partial \{f(x',\cdot)\leq h_1\}$ is non-empty. 
Then, from Lemma \ref{lem:distance from boundary},
\[
|f(x) - \tilde{f}^\tau(x)|
    = (h_2-h_1)
    \leq  c_0 \dist(\partial \{f(x',\cdot)\geq h_2\}, \partial \{f(x',\cdot)\leq h_1\})
\]
and with Lemma \ref{lem:distance for fixed h}, we obtain 
\begin{align*}
|f(x&) - \tilde{f}^\tau(x)|\\
    &\leq c_0 \dist(\partial \{f(x',\cdot)\geq h_2\}, \partial \{f(x',\cdot)\leq h_1\})\\
    &= c_0 \dist( \partial\{f(x',\cdot) >h_2\},  \partial\{f(x',\cdot)> h_1\})\\
    &\leq c_0 (\dist( \partial\{f(x',\cdot)> h_2\}, \partial\{\tilde{f}^\tau(x',\cdot)> h_1\})
         +\dist(  \partial\{\tilde{f}^\tau(x',\cdot)> h_1\},\partial\{f(x',\cdot)> h_1\}) )\\
    &\leq c_0 (0 + v_0(h_1)\tau) \leq c_0 \tau.
\end{align*}
Hence, the $L^\infty$ estimate holds. 
With Proposition \ref{prop:Lip constant}, we conclude that
\begin{align*}
\|f - \tilde{f}^\tau\|_{L^1(\R^n)}
    &= \int_{\supp f} |f(x) - \tilde{f}^\tau(x)| \, dx 
    \leq c_0 \tau |\supp f|. 
\end{align*}
\end{proof}

We conclude this section with an estimate proving that the $H^s$ norms of $f^\tau$ and $\tilde{f}^\tau$ can be made arbitrarily close for sufficiently small $h_0>0$ and in the case  $0 < s < 1/2$.  

\begin{lem}\label{lem:Hs-norms close}
Let $0 < s < 1/2$.  Assume that $f$ is Lipschitz with $[f]_{\text{Lip}(\R^n)} \leq c_0$ and not radially decreasing across $\{x_n=0\}$. 
Then, for any $\varepsilon>0$, there is a $h_0 = h_0(\varepsilon, n, s,f)>0$ and $\tau_0 = \tau_0(h_0,f)>0$ such that
\[
\big| \|\tilde{f}^\tau\|_{H^s(\R^n)}^2 -\|{f}^\tau\|_{H^s(\R^n)}^2 \big| < \varepsilon \tau \quad \hbox{for all}~0 < \tau \leq \tau_0.
\]
\end{lem}

\begin{proof}
Fix $\varepsilon>0$, and let $h_0>0$ to be determined. 
From Proposition \ref{prop:Lip constant}, we have
\[
[\tilde{f}^\tau]_{\text{Lip}(\R^n)}\leq \frac{c_0h_0}{h_0-c_0\tau} \leq 2c_0 \quad \hbox{for all}~0 \leq \tau \leq \frac{h_0}{2c_0} =: \tau_0. 
\]
Moreover, $\tilde{f}^\tau = f^\tau$ in $\{f^\tau>h_0\} \cup( \{\tilde{f}^\tau = 0\} \cap \{{f}^\tau = 0\})$, so that
\begin{align*}
\|\tilde{f}^\tau\|_{H^s(\R^n)}^2 -\|{f}^\tau\|_{H^s(\R^n)}^2
    &= \int_{\R^n} \int_{A_{h_0}} \frac{(\tilde{f}^\tau(x)-\tilde{f}^\tau(y))^2 -  ({f}^\tau(x)-{f}^\tau(y))^2}{|x-y|^{n+2s}} \, dx \, dy
\end{align*}
where 
\[
A_{h_0} :=( \{\tilde{f}^\tau<h_0\} \cup \{f^\tau<h_0\} )\cap (\supp \tilde{f}^\tau \cup \supp {f}^\tau).
\]
That is, $A_{h_0}$ is the set in which $\tilde{f}^\tau \not= f^\tau$.
To estimate the integral, we split into short and long-range interactions. 
Let $R>0$ be such that $\supp f^\tau \cup \supp \tilde{f}^\tau \subset B_R$ and write
\begin{align*}
\|\tilde{f}^\tau\|_{H^s(\R^n)}^2 -\|{f}^\tau\|_{H^s(\R^n)}^2
    &=  \int_{A_{h_0}} \int_{|x-y|<R} \frac{(\tilde{f}^\tau(x)-\tilde{f}^\tau(y))^2 -  ({f}^\tau(x)-{f}^\tau(y))^2}{|x-y|^{n+2s}} \, dy \, dx\\
 &\quad +\int_{A_{h_0}}\int_{|x-y|\geq R}  \frac{(\tilde{f}^\tau(x)-\tilde{f}^\tau(y))^2 -  ({f}^\tau(x)-{f}^\tau(y))^2}{|x-y|^{n+2s}} \, dy \, dx\\
 &=: I + II. 
\end{align*}

First considering $I$, 
notice that the support of the integrand is contained in the set
\begin{align*}
\{(x,y): |x|<R~\hbox{or}&~|y|<R\} \cap \{(x,y): |x-y|<R\} \\
 &\subset \{(x,y) : |x| <2R~\hbox{and}~|y|<2R\}\cap  \{(x,y): |x-y|<R\}.
\end{align*}
Therefore,
\begin{align*}
I
    &= \int_{A_{h_0} \cap B_{2R}} \int_{B_{2R} \cap B_R(x)} \frac{(\tilde{f}^\tau(x)-\tilde{f}^\tau(y))^2 -  ({f}^\tau(x)-{f}^\tau(y))^2}{|x-y|^{n+2s}} \, dy \, dx. 
\end{align*}
With the Lipschitz bounds for $\tilde{f}^\tau$ and $f^\tau$, we estimate
\begin{align*}
|(\tilde{f}^\tau(x)&-\tilde{f}^\tau(y))^2 -  ({f}^\tau(x)-{f}^\tau(y))^2|\\
    &=|(\tilde{f}^\tau(x)-\tilde{f}^\tau(y)) +  ({f}^\tau(x)-{f}^\tau(y))|
    |(\tilde{f}^\tau(x) -  {f}^\tau(x))-(\tilde{f}^\tau(y)-{f}^\tau(y))|\\
    &\leq ([\tilde{f}^\tau]_{\text{Lip}(\R^n)} + [{f}^\tau]_{\text{Lip}(\R^n)})|x-y|
    |(\tilde{f}^\tau(x) -  {f}^\tau(x))-(\tilde{f}^\tau(y)-{f}^\tau(y))|\\
    &\leq 3c_0|x-y| |(\tilde{f}^\tau(x) -  {f}^\tau(x))-(\tilde{f}^\tau(y)-{f}^\tau(y))|. 
\end{align*}
Therefore,
\begin{equation}\label{eq:1}
\begin{aligned}
|I|
    &\leq 3c_0   \int_{A_{h_0} \cap B_{2R}} \int_{B_{2R}\cap B_R(x)}\frac{|(\tilde{f}^\tau(x) -  {f}^\tau(x))-(\tilde{f}^\tau(y)-{f}^\tau(y))|}{|x-y|^{n+2s-1}} \, dy \, dx \\
    &\leq 3c_0 \bigg[ 
    \int_{A_{h_0} \cap B_{2R}} |\tilde{f}^\tau(x) -  {f}^\tau(x)|\bigg(\int_{B_{2R} \cap B_R(x)}\frac{1}{|x-y|^{n+2s-1}} \, dy \bigg) dx \\
    &\quad+
    \int_{B_{2R}}|\tilde{f}^\tau(y)-{f}^\tau(y)|\bigg(\int_{A_{h_0} \cap B_{2R} \cap B_R(y)} \frac{1}{|x-y|^{n+2s-1}} \, dx \bigg) \, dy \bigg].
\end{aligned}
\end{equation}
Since $0< s < 1/2$, we have
\begin{equation}\label{eq:2}
\begin{aligned}
\int_{B_{2R} \cap B_R(x)}\frac{1}{|x-y|^{n+2s-1}} \, dy 
    &\leq \int_{B_{R(x)}}\frac{1}{|x-y|^{n+2s-1}} \, dy \\
    &= \int_{B_R}\frac{1}{|z|^{n+2s-1}} \, dz 
    = c_{n,s,R}< \infty
\end{aligned}
\end{equation}
and similarly
\begin{equation}\label{eq:3}
\int_{A_{h_0} \cap B_{2R} \cap B_R(y)} \frac{1}{|x-y|^{n+2s-1}} \, dx
    \leq \int_{B_R(y)} \frac{1}{|x-y|^{n+2s-1}} \, dx = c_{n,s,R} < \infty.
\end{equation}
Using again that $\tilde{f}^\tau = f^\tau$ in $\R^n \setminus A_{h_0}$, we note that
\begin{equation}\label{eq:4}
 \int_{B_{2R}}|\tilde{f}^\tau(y)-{f}^\tau(y)| \, dy 
    =  \int_{A_{h_0} \cap B_{2R}}|\tilde{f}^\tau(y)-{f}^\tau(y)| \,dy. 
\end{equation}
Therefore, from \eqref{eq:1}, \eqref{eq:2}, \eqref{eq:3}, and \eqref{eq:4}, 
\begin{align*}
|I|
    &\leq C_{n,s,f,R} \bigg[ 
    \int_{A_{h_0} \cap B_{2R}} |\tilde{f}^\tau(x) -  {f}^\tau(x)| dx+
    \int_{B_{2R}}|\tilde{f}^\tau(y)-{f}^\tau(y)| \, dy \bigg]\\
    &\leq 2C_{n,s,f,R}  
    \int_{A_{h_0} } |\tilde{f}^\tau(x) -  {f}^\tau(x)| dx.
\end{align*}
By Lemma \ref{lem:Lp bound}, we arrive at
\begin{align*}
|I|
    &\leq C_{n,s,f}  |A_{h_0}| \tau.
\end{align*}

Regarding $II$, we expand the squares to obtain 
\begin{equation}\label{eq:II1}
\begin{aligned}
II
    &= \int_{A_{h_0}}\int_{|x-y|\geq R}  
     \frac{(\tilde{f}^\tau(x))^2- (f^\tau(x))^2}{|x-y|^{n+2s}} \, dy \, dx \\
     &\quad + \int_{A_{h_0}}\int_{|x-y|\geq R}  
     \frac{(\tilde{f}^\tau(y))^2- (f^\tau(y))^2}{|x-y|^{n+2s}} \, dy \, dx \\
     &\quad - 2 \int_{A_{h_0}}\int_{|x-y|\geq R}
    \frac{\tilde{f}^\tau(x)\tilde{f}^\tau(y) -  {f}^\tau(x){f}^\tau(y)}{|x-y|^{n+2s}} \, dy \, dx. 
\end{aligned}
\end{equation}
First observe that
\begin{equation}\label{eq:II2}
\begin{aligned}
\int_{A_{h_0}}\int_{|x-y|\geq R}  
     \frac{(\tilde{f}^\tau(x))^2- (f^\tau(x))^2}{|x-y|^{n+2s}} \, dy \, dx
     &= \int_{A_{h_0}}\big[(\tilde{f}^\tau(x))^2- (f^\tau(x))^2 \big]\bigg(\int_{|z|\geq R}  
     \frac{1}{|z|^{n+2s}} \, dy \bigg) \, dx\\
     &=C_{n,s,R} \big( \|\tilde{f}^\tau\|_{L^2(A_{h_0})}^2-\|f^\tau\|_{L^2(A_{h_0})}^2\big)
\end{aligned}
\end{equation}
and similarly, using Lemma~\ref{lem:Lp bound},
\begin{equation}\label{eq:II3}
\begin{aligned}
\int_{A_{h_0}}\int_{|x-y|\geq R}  
     &\frac{(\tilde{f}^\tau(y))^2- (f^\tau(y))^2}{|x-y|^{n+2s}} \, dy \, dx 
     = \int_{A_{h_0}}\int_{|x-y|\geq R}  
     \frac{(\tilde{f}^\tau(y))^2- (f^\tau(y))^2}{|x-y|^{n+2s}} \, dx \, dy\\
     &= \int_{A_{h_0}}\big[(\tilde{f}^\tau(y)- f^\tau(y))(\tilde{f}^\tau(y)+ f^\tau(y)) \big] \bigg(\int_{|z|\geq R}  
     \frac{1}{|z|^{n+2s}} \, dz \bigg) \, dy\\
     &\le C_{n,s,R}\|f\|_{L^\infty} \|\tilde{f}^\tau-f^\tau\|_{L^1(A_{h_0})}\\
     &\le C_{n,s,f,R}|A_{h_0}| \tau.
\end{aligned}
\end{equation}
Consequently, from \eqref{eq:II1}, \eqref{eq:II2}, and \eqref{eq:II3},
\begin{align*}
|II|
    &\leq   C_{n,s,f,R}|A_{h_0}| \tau
    + \bigg|\int_{A_{h_0}}
    \tilde{f}^\tau(x)(W_R*\tilde{f}^\tau)(x) -  {f}^\tau(x)(W_R*{f}^\tau)(x) \, dy \, dx\bigg|
\end{align*}
where $W_R(x) = |x|^{-n-2s} \chi_{\R^n\setminus B_R(0)}(x)$. 
Using that
\[
|\nabla (W_R*\tilde{f}^\tau)(x)| 
    \leq \int_{\R^n} W_R(y) |\nabla \tilde{f}^\tau(x-y)| \, dy 
    \leq 2c_0 \int_{\R^n} W_R(y) \, dy = c_{n,s,f,R},
\]
and similarly for $|\nabla W_R*f^\tau|$, we can follow the proof of \cite{CHVY}*{Proposition 2.8} to show that
\begin{align*}
\bigg|\int_{A_{h_0}}
    \tilde{f}^\tau(x)(W_R*\tilde{f}^\tau)(x) -  {f}^\tau(x)(W_R*{f}^\tau)(x) \, dy \, dx\bigg|
        \leq C_{n,s,f,R} \|\min\{f, h_0\}\|_{L^1(\R^n)} \tau.
\end{align*}
Summarizing, we have
\begin{align*}
\left|\|\tilde{f}^\tau\|_{H^s(\R^n)}^2 -\|{f}^\tau\|_{H^s(\R^n)}^2\right|
    &\leq |I| + |II|\\
    &\leq C_{n,s,f,R}  \bigg(|A_{h_0}| + \|\min\{f, h_0\}\|_{L^1(\R^n)}\bigg) \tau \\
    &< \varepsilon \tau
\end{align*}
for $h_0$ sufficiently small. 
\end{proof}

\subsection{Proof of Theorem \ref{thm:uniqueness}}

The proof relies on two results regarding the nonlocal energy 
\[
    \mathcal{E}_s(v)=c_{n,s}[v]_{H^s}^2+\int_{\R^n} |x|^2v(x)\;dx
\]
where we recall that $c_{n,s}[v]_{H^s}^2 =\langle (-\Delta)^s v, v \rangle_{L^2(\R^n)}$. 
First, we will show that small perturbations of stationary solutions $v$ to the fractional thin film equation that preserve the support of $v$ correspond to small perturbations in the energy.

\begin{lem}\label{lem:perturbation}
Assume that $v$ is Lipschitz and satisfies the stationary equation
\begin{equation}\label{eq:EL}
\begin{cases}
(-\Delta)^s v = \sum_i \lambda_i \chi_{\mathcal{P}_i}(y) - \frac{\beta}{2} |y|^2 & \hbox{in}~\supp(v) \subset \R^n\\
v \geq 0 & \hbox{in}~\R^n.
\end{cases}
\end{equation}
Let $v^\tau$ be a perturbation of $v$ which is continuous in the $C^\alpha$ norm for every $0 <\alpha<1$ and preserves mass in each connected component. Then
    $$
    \lim_{\tau\to 0^+}\frac{\mathcal{E}_s(v^\tau)-\mathcal{E}_s(v)}{\tau}=0.
    $$
\end{lem}

\begin{proof}
From the definition of $\mathcal{E}_s$ and that $(-\Delta)^s$ is self-adjoint, we have
\begin{align*}
\frac{\mathcal{E}_s(v^\tau)-\mathcal{E}_s(v)}{\tau}
    &=\int_{\R^n}  \bigg(\frac{1}{2}(-\Delta)^s(v^\tau+v)+\frac{\beta}{2}|y|^2\bigg) \frac{(v^\tau-v)}{\tau} \, dy\\
    &=\int_{\R^n}  \bigg((-\Delta)^sv+\frac{\beta}{2}|y|^2\bigg) \frac{(v^\tau - v)}{\tau} \, dy +\frac{1}{2}\int_{\R^n} (-\Delta)^s(v^\tau-v) \frac{(v^\tau-v)}{\tau} \, dy.
\end{align*}
Using that $v$ solves \eqref{eq:EL} and that $v^\tau$ preserves the mass of $v$ in each $\mathcal{P}_i$, we notice that the first term vanishes
$$
\int_{\R^n}  \bigg((-\Delta)^sv+\frac{\beta}{2}|y|^2\bigg) \frac{(v^\tau - v)}{\tau} \, dy =\sum_i \lambda_i \frac{1}{\tau}\int_{\mathcal{P}_i}(v^\tau-v) \, dy =0. 
$$
Using Lemma \ref{lem:Lp bound} and that
$$
(-\Delta)^s(v^\tau-v)\to 0\qquad \mbox{in $C^0(\R^n)$}
$$
we take $\tau\to 0^+$ in the second term to complete the proof. 
\end{proof}

Next, we show that if the perturbation is precisely the truncated Steiner symmetrization of $v$, then the energy is in fact strictly decreasing. 

\begin{prop}\label{prop:decreasingE}
Assume $0< s<1/2$, and $v$ is Lipschitz, non-negative with compact support. If $v$ is not radially decreasing, then there exist constants $h_0,\gamma, \tau_0>0$ such that
\[
\mathcal{E}_s(\tilde{v}^{\tau}) \leq \mathcal{E}_s(v) - \frac{ c_{n,s}}{2}\gamma \tau \quad \hbox{for all}~ 0 < \tau < \tau_0
\]
where $\tilde{v}^{\tau}$ is the continuous Steiner symmetrization truncated at height $h_0$.
\end{prop}

\begin{proof}
Begin by writing 
\begin{align*}
\mathcal{E}_s(\tilde{v}^\tau)
     - \mathcal{E}_s(v)
     &= c_{n,s}\big([\tilde{v}^\tau]_{H^s(\R^n)} - [v]_{H^s(\R^n)}\big)
             + \frac{\beta}{2}\int_{\R^n}\abs{y}^2(\tilde{v}^\tau - v) \, dy.
\end{align*}
Re-arranging, we have
\begin{align*}
\int_{\R^n}\abs{y}^2(\tilde{v}^\tau(y)-v(y)) \, dy
    &= \int_{\R^n}\abs{y}^2 \int_0^\infty (\chi_{M^{v_0(h)\tau}(U_{y'}^h)}(y_n)- \chi_{U_{y'}^h}(y_n))\, dh \,  dy\\
     &= \int_{\R^{n-1}} \int_0^{\infty} \bigg(\int_{M^{v_0(h)\tau}(U_{y'}^h)} \abs{y_n}^2  \, dy_n - \int_{U_{y'}^h} |y_n|^2 \, dy_n \bigg) \, dh \, dy' \\
    &\leq0,
\end{align*}
where the last inequality follows by the definition of the symmetrization. On the other hand, by Theorem \ref{thm:main2} and Lemma \ref{lem:Hs-norms close} with $\varepsilon = \gamma/2$, there are $h_0,\tau_0>0$ such that 
\begin{align*}
[\tilde{v}^\tau]^2_{H^s(\R^n)} - [v]^2_{H^s(\R^n)}
    &=  \big([\tilde{v}^\tau]^2_{H^s(\R^n)} -  [v^\tau]^2_{H^s(\R^n)}\big)
            + \big([v^\tau]^2_{H^s(\R^n)} - [v]^2_{H^s(\R^n)}\big)\\
    &\leq \frac{\gamma}{2}\tau - \gamma\tau = - \frac{\gamma}{2}\tau\qquad\mbox{for any $\tau<\tau_0$.}
\end{align*}
\end{proof}

\begin{proof}[Proof of  Theorem \ref{thm:uniqueness}]
Assume, by way of contradiction, that $v$ is not radially decreasing and let $\tilde{v}^\tau$ denote the continuous Steiner symmetrization of $v$ truncated at height $h_0 >0$. 
By Proposition \ref{prop:Lip constant}, we have that $\tilde{v}^\tau$ is Lipschitz for sufficiently small $0 \leq \tau < c_0/h_0$ and preserves the mass of $v$ in each connected component. 
Moreover, by Lemma \ref{lem:Lp bound},  $\tilde{v}^\tau$ is continuous in $\tau$ in the $C^\alpha$ norm for any $0 < \alpha <1$. 
Hence, the hypotheses of Lemma \ref{lem:perturbation} are satisfied so that, for all $\varepsilon>0$, there is $\delta>0$ such that
\[
 - \varepsilon \tau < \mathcal{E}_s(\tilde{v}^\tau) - \mathcal{E}(v)< \tau \varepsilon \quad \hbox{for all}~ 0 \leq \tau < \delta. 
\]
On the other hand, Proposition \ref{prop:decreasingE} guarantees that
\[
\mathcal{E}_s(\tilde{v}^\tau) - \mathcal{E}(v) < - \frac{c_{n,s}}{2}\gamma \tau
\quad \hbox{for all}~0 \leq \tau < \tau_0.
\]
We arrive at a contradiction by choosing $0 < \varepsilon < c_{n,s}\gamma/2$. Therefore, it must be that $v$ is radially decreasing. 
Consequently, $\supp v$ is a single connected component. 

To show uniqueness, up to the scaling, we follow the argument in the proof of \cite{delgadino2022uniqueness}*{Theorem 1.1}. 
Consider two radially symmetric critical points $v_0, v_1$ that are Lipschitz, and let $\{v_t\}_{t\in[0,1]}$ be the height function interpolation presented in Section~\ref{sec:interp}. 
Using that $v_0,v_1 \in C^{0,1}(\R^n)$, we can use Proposition~\ref{prop:interW1p} to conclude that $\{v_t\}_{t\in[0,1]}$ is continuous in $C^{\alpha}(\R^n)$ for any $0<\alpha<1$. 
Recall the upper and lower Dini derivatives in \eqref{eq:dini}. 
We claim that
$$
\frac{d^+}{dt}\mathcal{E}_s(v_t)\bigg|_{t=0}=\lim_{t\to 0^+}\frac{\mathcal{E}_s(v_t) - \mathcal{E}(v_0)}{t}=0=\lim_{t\to 1^-}\frac{\mathcal{E}_s(v_t) - \mathcal{E}(v_1)}{1-t}=\frac{d^-}{dt}\mathcal{E}_s(v_t)\bigg|_{t=1}.
$$ 
Following the proof of \cite{delgadino2022uniqueness}*{Proposition 4.4}, it is enough to show that
\begin{equation}\label{eq:t-limit}
\frac{d^+}{dt}[v_t]_{H^s}^2\bigg|_{t=0} 
=\int_{\R^n} (-\Delta)^s v_0 \frac{d v_\tau}{d \tau} \bigg|_{\tau=0}\;dx
\end{equation}
as the potential part of the energy follows in the same way. Since $(-\Delta)^s$ is self-adjoint,
$$
\frac{d^+}{dt}[v_t]_{H^s}^2\bigg|_{t=0} 
=\lim_{t\to 0^+}\frac{1}{t}\int_0^t\int_{\R^n} (-\Delta)^s v_\tau \frac{d v_\tau}{d \tau}\;dx\,d\tau.
$$
Then, \eqref{eq:t-limit} holds as long as the pairing
$$
\int_{\R^n} (-\Delta)^s v_\tau \frac{d v_\tau}{d \tau}\;dx
$$
is continuous in $\tau$. From the Lipschitz a priori estimate, we know that
$$
(-\Delta)^s v_\tau\to (-\Delta)^s v_0\qquad\mbox{strongly in continuous functions as}~\tau \to0,
$$
so we only need to check that weakly
$$
\frac{d v_\tau}{d \tau}\rightharpoonup \frac{d v_\tau}{d \tau}\bigg|_{\tau=0}.
$$
This follows directly from \cite{delgadino2022uniqueness}*{Lemma 4.3}, after noticing that both $v_0$ and $v_1$ are not degenerate. More specifically, $v_0$ and $v_1$ are twice differentiable around zero, and there exists a $c>0$ such that
$$
\max\{\Delta v_0(0),\Delta v_1(0)\}<-c.
$$
This follows because both $v_0$ and $v_1$ solve a fractional elliptic equation in a neighborhood of zero.


However, by the strict convexity of $\mathcal{E}_s(v_t)$, see Theorem~\ref{thm:nonlocal-convexity} and Proposition~\ref{prop:V}, we know  that
$$
\frac{d^+}{dt}\mathcal{E}_s(v_t)\bigg|_{t=0}<\frac{d^-}{dt}\mathcal{E}_s(v_t)\bigg|_{t=1},
$$
which is a contradiction. Therefore, for any given mass there is a unique critical point to $\mathcal{E}_s$, and it is given by \eqref{eq:explicitv}, see \cite{Dyda}.
\end{proof} 

\section*{Acknowledgements}

MGD and MV acknowledge the support of NSF-DMS RTG 18403. MGD research is partially supported by NSF-DMS 2205937. 
MV acknowledges the support of Australian Laureate Fellowship FL190100081 ``Minimal surfaces, free boundaries and partial differential equations.'' 
The authors would like to thank Yao Yao for discussions and encouragement in the early stages of this project. 

\bigskip

\noindent {\bf Data Availability:} Data sharing not applicable to this article as no datasets were generated or analysed during the current study.

\medskip

\noindent {\bf Compliance with ethical standards}

\medskip

\noindent {\bf Conflict of interest:} The authors declare that they have no conflict of interest.

\bibliographystyle{imsart-number}
\bibliography{ref}
\end{document}